\documentclass{article}
\usepackage{hyperref}
\usepackage{graphicx}
\usepackage[all]{xy} 
\usepackage{amsmath}
\usepackage{amsthm}
\usepackage{array}  
\usepackage{amssymb}
\usepackage{calc}
\usepackage{txfonts}
\usepackage{textcomp}
\usepackage{colortbl}
\usepackage{tabularx} 
\usepackage[table]{xcolor}
\usepackage{tensor}
\usepackage{appendix} 
\usepackage{placeins} 
\usepackage[T1]{fontenc} 
\usepackage{titletoc}
\usepackage{pgf}
\usepackage{xcolor,lipsum}
\usepackage{tikz}
\usetikzlibrary{matrix,arrows}
\usetikzlibrary{patterns}
\usetikzlibrary{shapes,decorations,shadows}

\DeclareMathOperator{\Sym}{Sym} 
\DeclareMathOperator{\sym}{sym} 
\DeclareMathOperator{\SRep}{srep} 
\DeclareMathOperator{\SL}{SL} 
\DeclareMathOperator{\SP}{SP} 
\DeclareMathOperator{\SO}{SO} 
\DeclareMathOperator{\SR}{SR} 
\DeclareMathOperator{\R}{R} 
\DeclareMathOperator{\Or}{O} 
\DeclareMathOperator{\Lie}{Lie} 
\DeclareMathOperator{\Solp}{Solp} 
\DeclareMathOperator{\Oolp}{Oolp} 
\DeclareMathOperator{\End}{End} 
\DeclareMathOperator{\Aut}{Aut} 
 
\DeclareMathOperator{\Hom}{Hom} 
\DeclareMathOperator{\rad}{rad}

\DeclareMathOperator{\rep}{rep}

\DeclareMathOperator{\GL}{GL} 
 
\DeclareMathOperator{\stab}{stab}

\definecolor{lblue}{rgb}{0.3,0.0,4.4}
\newcommand*{\punkte}{\dots\unkern}

\newcommand{\A}{\mathcal{A}}

\newcommand{\Q}{\mathcal{Q}} 
\newcommand{\Orb}{\mathcal{O}} 
\newcommand{\N}{\mathcal{N}}

\newcommand{\gl}{\mathfrak{gl}}

\newcommand{\fg}{{\mathfrak g}}
\newcommand{\fb}{{\mathfrak b}}

\newcommand{\fsp}{\mathfrak{sp}}
\newcommand{\fo}{{\mathfrak o}}
\newcommand{\fsym}{\mathfrak{sym}}

\newcommand{\dimv}{\underline{\dim}}
 
\newcommand{\df}{\underline{d}} 

\newcommand{\form}{\langle-,-\rangle}

\def\presuper#1#2%
  {\mathop{}%
   \mathopen{\vphantom{#2}}^{#1}%
   \kern-\scriptspace%
   #2}

\numberwithin{equation}{section}

\newtheorem{theorem}{Theorem}[section]
\newtheorem{lemma}[theorem]{Lemma}
\newtheorem{definition}[theorem]{Definition}
\newtheorem{proposition}[theorem]{Proposition}
\newtheorem{corollary}[theorem]{Corollary}

\newtheorem{example}[theorem]{Example}

\newtheorem{convention}[theorem]{Convention}
\theoremstyle{remark}
\newtheorem{remark}[theorem]{Remark}

\author{Magdalena Boos \thanks{Ruhr University Bochum, Faculty of Mathematics,  44780 Bochum, Germany. magdalena.boos-math@rub.de}~,~Giovanni Cerulli Irelli \thanks{Sapienza University of Rome, Department SBAI, 00161 Rome, Italy. giovanni.cerulliirelli@uniroma1.it}~,~ Francesco Esposito \thanks{University of Padova, Department of Mathematics, 35121 Padova, Italy. esposito@math.unipd.it}} 

\makeatletter
\long\def\nnfoottext#1{\insert\footins{\footnotesize
    \interlinepenalty\interfootnotelinepenalty
    \splittopskip\footnotesep
    \splitmaxdepth \dp\strutbox \floatingpenalty \@MM
    \hsize\columnwidth \@parboxrestore
   \edef\@thefnmark{}
   \edef\@currentlabel{}\@makefntext
    {\rule{\z@}{\footnotesep}\ignorespaces
      #1\strut}}}
\makeatother
\begin{document}
\parindent0pt

\makeatletter
\let\@fnsymbol\@arabic
\makeatother
\title{\bf Parabolic orbits of $2$-nilpotent elements for classical groups}
\date{}
\maketitle
\begin{abstract}
We consider the conjugation-action of the Borel subgroup of the symplectic or the orthogonal group on the variety of nilpotent complex elements of nilpotency degree $2$ in its Lie algebra. We translate the setup to a representation-theoretic context in the language of a symmetric quiver algebra. This makes it possible to provide a parametrization of the orbits via a combinatorial tool that we call symplectic/orthogonal oriented link patterns. We deduce information about numerology. We then generalize these classifications to standard parabolic subgroups for all classical groups. Finally, our results are restricted to the nilradical.
\end{abstract}
\section{Introduction}\label{intro}
Let $G$ be a classical complex group of rank $n$. Then G is either the general linear group $\GL_n(K)$ or the symplectic group $\SP_{2l}(K)$ or the orthogonal group $\Or_n(K)$, where $K=\mathbf{C}$. Let $\mathfrak{g}$ be the corresponding Lie algebra.

The study of the adjoint action of (subgroups of) $G$ on $\mathfrak{g}$ and numerous variants thereof is a well-established and much considered task in algebraic Lie theory. Employing methods of geometric invariant theory, a classical topic is the study of orbits and their closures, which is also known as the \emph{vertical} problem \cite{Kr}. 

One famous example of a classification problem alike is the study of $\GL_n$-conjugation (or $\SL_n$-conjugation, this doesn't make a difference) on the variety of complex matrices of square--size $n$.  A complete system of representatives up to conjugation is given by the  Jordan canonical form \cite{Jo2} which dates back to the 19\textsuperscript{th} century. This system of representatives is given by continuous parameters, the eigenvalues of the matrix, and discrete parameters. In order to determine the latter, it suggests itself to restrict the action to the nilpotent cone, namely to $\GL_n$-conjugation on the set of nilpotent matrices. The number of conjugacy classes of nilpotent matrices is finite and can be described combinatorially by partitions of $n$.

One generalization of this setup is obtained by restricting the acting group from $G$ to parabolic subgroups $P\subseteq G$. In particular, the Borel subgroup $B$ is considered, then, and the question about a variety admitting only finitely many orbits is closely related to the concept of so-called spherical varieties \cite{Br}. One example of  a  parabolic action can be found in \cite{HiRoe}, where Hille and R\"ohrle prove a finiteness criterion for the number of orbits of parabolic conjugation on the unipotent radical of $\mathfrak{g}$.

Another adaption of the above setup is given by restricting the nilpotent cone $\N$ of nilpotent matrices to certain subvarieties. For example, Melnikov parametrizes the Borel-orbits in the variety of $2$-nilpotent elements in the nilradical $\mathfrak{n}$ of $\mathfrak{g}=\Lie(\GL_n(K))$ in \cite{Me1} which is inspired by the study of orbital varieties. A parametrization in the symplectic setup is published by Barnea and Melnikov in \cite{BaMe}. In \cite{GMMP}, Gandini, Maffei, M\"oseneder Frajria and Papi consider the more general approach of $B$-stable abelian subalgebras of the
nilradical of $\mathfrak{b}$ in which they parametrize the $B$-orbits and describe their closure
relations.

In this article, we consider the algebraic subvariety $\N(2)$ of $2$--nilpotent elements of the nilpotent cone of $\mathfrak{g}$, namely
$$
\N(2)=\N(2,G)=\{x\in\mathfrak{g}|\,x^2=0\}.
$$
Every parabolic subgroup $P$ of $G$ acts on $\N(2)$. It is known that the number of orbits is always finite, since Panyushev shows finiteness for the Borel-action in \cite{Pan2}. In case $G=\GL_n(K)$, a parametrization of the $P$-orbits  and a description  of their degenerations is given in \cite{BoRe} and \cite{B2} for each parabolic subgroup $P\subset G$.

Our first goal in this article is to prove in a different manner that there are only finitely many $B$-orbits in $\N(2)$ for the remaining classical groups, that is, for types $B$, $C$ and $D$. We approach the problem in a way closely related to \cite{BoRe} from a quiver-theoretic point of view - but instead of translating to the representation variety of a quiver with relations of a special dimension vector, we translate the orbits to certain (isomorphism classes of) representations of a \textit{symmetric} quiver with relations of a fixed dimension vector. In this setup we show that there are only finitely many of the latter which are parametrized by combinatorial objects which we call symplectic/orthogonal oriented  link patterns, see Definitions~\ref{Def:SOLP} and \ref{Def:OOLP}.

For example the Borel-orbits of $2$-nilpotent matrices in $\mathfrak{o}_4$ are parametrized by these five patterns:
$$
\tiny
\begin{tabular}{|c|c|c|c|c|}
\hline&&&&\\[-1ex]
\xymatrix{\textrm{\tiny{1}}\ar@/^1pc/[r]&
\textrm{\tiny{2}}}
&\xymatrix{\textrm{\tiny{1}}&
\textrm{\tiny{2}}\ar@/_1pc/[l]}
&\xymatrix{\textrm{\tiny{1}}\ar@/^1pc/[r]|{\bullet}&
\textrm{\tiny{2}}}
&\xymatrix{\textrm{\tiny{1}}&
\textrm{\tiny{2}}\ar@/_1pc/[l]|{\bullet}}&\xymatrix{\textrm{\tiny{1}}&
\textrm{\tiny{2}}}
\\\hline
\end{tabular}
$$

The following five matrices give a system of representatives of these orbits.
$$
\tiny
\begin{tabular}{|c|c|c|c|c|}\hline
&&&&\\[-1ex]
$\left(\begin{array}{cccc}
0&0&0&0\\
1&0&0&0\\
0&0&0&0\\
0&0&-1&0\end{array}\right)$ 
& $\left(\begin{array}{cccc}
0&1&0&0\\
0&0&0&0\\
0&0&0&-1\\
0&0&0&0\end{array}\right)$ 
&$\left(\begin{array}{cccc}
0&0&0&0\\
0&0&0&0\\
1&0&0&0\\
0&-1&0&0\end{array}\right)$ 
&$\left(\begin{array}{cccc}
0&0&1&0\\
0&0&0&-1\\
0&0&0&0\\
0&0&0&0\end{array}\right)$ 
&$\left(\begin{array}{cccc}
0&0&0&0\\
0&0&0&0\\
0&0&0&0\\
0&0&0&0\end{array}\right)$\\\hline
\end{tabular}
$$

Our second goal is to parametrize all orbits explicitly. The approach via a symmetric quiver makes it possible to classify the orbits by  representations, and thus, by combinatorial data. 

We afterwards generalize these results to parabolic subgroups $P\subset G$. To do so, we consider the isotropic flag corresponding to $P$ and realize it as a representation $M_P$. The parabolic $P$ equals the symmetric stabilizer of $M_P$ and we obtain a classification of the $P$-orbits in $\N(2)$ as in the Borel case.

In the last section, we restrict our results to the action of $P$ on the nilradical $\mathfrak{n}(2):=\N(2)\cap \mathfrak{n}$ of $2$-nilpotent upper-triangular matrices in $\fg$ and obtain complete parametrizations, here. 
In the symplectic case, the parametrization coincides with the parametrization by so-called symplectic link patterns of \cite{BaMe}, even though the methods used to prove it are different.  

{\bf Acknowledgments:} The authors would like to thank Giovanna Carnovale for her input concerning the results and methods of this work. Furthermore, the first author thanks Martin Bender for many discussions about Lie-theoretical background.

 \section{Classical groups and Lie algebras}\label{sect:classicalLie}
Let  $K$ be the field of complex numbers $K:=\mathbf{C}$ and let $n$ be an integer. We consider the complex classical groups, that is, the general linear group $\GL_n:=\GL_n(K)$, the symplectic group $\SP_{n}:=\SP_{n}(K)$, whenever $n=2l$ for some integer $l$, and the orthogonal group $\Or_n:=\Or_n(K)$. The corresponding Lie algebras are denoted by $\gl_n:=\gl_n(K)$, $\fsp_n:=\fsp_n(K)$ and $\fo_n:=\fo_n(K)$.

In general, given a vector space $V$ endowed with a non-degenerate bilinear form $\form$, let us denote by $\Sym(V)$ the group of symmetries of the vector space $V$ which preserve $\form|_{{V}\times V}$. Then $\Sym(V)$ equals either the symplectic group $\SP(V)$ or the orthogonal group $\Or(V)$, depending on whether $(V,\form)$ is symplectic or orthogonal). We define $\fsym(V):=\Lie (\Sym(V))$.

Let $l$ be an integer, then we denote by $J=J_l$ the $l\times l$ anti-diagonal matrix with every entry on the anti-diagonal being $1$: 
$$
J_l=\left(\vcenter{\xymatrix{0\ar@{..}[r]\ar@{..}[d]&1\ar@{..}[dl]\ar@{..}[d]\\1\ar@{..}[r]&0}}\right)
$$
It is easy to see (and well--known) that $J^{-1}=J$ and that  the conjugate $J\presuper{T\!\!\!}{A}J$ by $J$ of the \textit{transpose} $\presuper{T\!\!\!}{A}$ of a  matrix $A\in K^{l\times l}$ is given by "the transpose of $A$ with respect to the anti-diagonal". For example, for $l=2$, given 
$A=\left[\begin{array}{cc}a&b\\c&d\end{array}\right]:$ 
\[
J \presuper{T\!\!\!}{A}J=\left[\begin{array}{cc}
0&1\\1&0
\end{array}\right]\left[
\begin{array}{cc}
a&c\\b&d
\end{array}\right]\left[
\begin{array}{cc}
0&1\\1&0
\end{array}\right]
=\left[
\begin{array}{cc}
d&b\\c&a
\end{array}\right].
\]
We set 
\[\presuper{\mathfrak{T}}{A}:=J \presuper{T\!\!\!}{A}J.\]

In this notation, it is easy to write down the elements of the symplectic and orthogonal Lie algebras.

\subsection{Symplectic group}
Let $V$ be an $n=2l$--dimensional complex vector space. Let us fix a basis of $V$ and a bilinear form $F=F_V:V\times V\rightarrow K$, $F(v,w)=\langle v,w\rangle$, associated with the matrix (still denoted by $F$)
\begin{equation}\label{Eq:FormSymplectic}
F=\left[\begin{array}{cc}0&J_l\\-J_l&0\end{array}\right].
\end{equation}

The symplectic group $\SP_{n}$ consists of those matrices $A\in \GL_{n}$ which preserve this bilinear form (i.e. $\langle Av,Aw\rangle=\langle v,w\rangle$); in other words $A$ satisfies the equation $$\presuper{T\!}{A}FA=F.$$ 
The Lie algebra $\fsp_{n}$ of $\SP_{n}$ consists of those matrices $a\in \gl_{n}$ which fulfill
\begin{equation}\label{Eq:LieAlge}
\presuper{T\!}{a}F+Fa=0.
\end{equation} We write  
$a=\left[\begin{array}{cc}A&B\\C&D\end{array}\right]$, where $A, B, C, D$ are $l\times l$-blocks, so that condition \eqref{Eq:LieAlge} translates into the following equations:
\begin{equation}\label{Eq:SympLie}
a=\left[\begin{array}{cc}A&B=\presuper{\mathfrak{T}}{B}\\C=\presuper{\mathfrak{T}}{C}&D=-\presuper{\mathfrak{T}}{A}\end{array}\right].
\end{equation}
In particular, $\fsp_{n}$ has dimension $l^2+l(l+1)=l(n+1)$. The intersection of $\fsp_{n}$ with the Borel subalgebra $\fb_{n}:=\fb_n(K)$ of upper-triangular matrices is a solvable subalgebra of $\fsp_{n}$ of dimension $l(l+1)=l^2+l$. Since $\fsp_{n}$ is a Lie algebra of type $C_l$, the number of positive roots is $l^2$ and the number of simple roots is $l$; we hence see that $\fb(\fsp_{n}):=\fsp_{n}\cap \fb_{n}$ is a solvable subalgebra of maximal dimension and hence a Borel subalgebra. This is one of the advantages of working with the form $F$ given by \eqref{Eq:FormSymplectic}.

\subsection{Orthogonal group}
Let $V$ be an $n$--dimensional complex vector space (where $n$ can be  even or odd). Let us fix a basis of $V$ and let us choose the non--degenerate bilinear form on $V$ associated with the matrix $F=J_n$. The orthogonal group $\Or_{n}$ consists of those matrices $A\in \GL_{n}$ for which $\presuper{T\!}{A}FA=F$ holds true. The Lie algebra $\fo_{n}$  consists of those matrices $a\in \gl_{n}$ satisfying \eqref{Eq:LieAlge} which translates into the relation
\begin{equation}\label{Eq:OrthLie}
a=-\presuper{\mathfrak{T}}{a}.
\end{equation}
In particular, $\fo_{n}$ has dimension $\frac{n(n-1)}{2}$. 
The intersection of $\fo_{n}$ with the Borel subalgebra $\fb_{n}$ of upper-triangular matrices, is a solvable subalgebra of $\fo_{n}$.
\begin{itemize}
\item If $n=2l$, the dimension of such a solvable subalgebra is easily seen to be 
\[\frac{n(n-1)}{2}-l(l-1)=l(2l-1)-l(l-1)=l^2.\] Since $\fo_{n}$ is a Lie algebra of type $D_l$, the number of positive roots is $l(l-1)$ and the number of simple roots is $l$; we hence see that $\fb(\fo_{n}):=\fo_{n}\cap \fb_{n}$ is a solvable subalgebra of maximal dimension and hence a Borel subalgebra.

\item Similarly, if $n=2l+1$, the dimension of $\fb(\fo_{n}):=\fo_{n}\cap \fb_{n}$ is easily seen to be 
\[\frac{n(n-1)}{2}-(l(l-1)+l)=(2l+1)l-l^2=l^2+l.\] Since $\fo_{n}$ is a Lie algebra of type $B_l$, the number of positive roots is $l^2$ and the number of simple roots is $l$; we hence see that $\fb(\fo_{n})$ is a solvable subalgebra of maximal dimension and hence a Borel subalgebra. 
\end{itemize}

As before, this is one of the benefits of working with the form F given by $J_n$.

\section{Background on (symmetric) quiver representations}
 We include basic knowledge about the representation theory of finite-dimensional algebras via finite quivers \cite{ASS} before introducing the notion of a symmetric quiver and discussing its representations. This theoretical background will be necessary later on to prove our main results.
 
A \textit{finite quiver} $\Q$ is a directed graph $\Q=(\Q_0,\Q_1,s,t)$, such that $\Q_0$ is a finite set of \textit{vertices} and $\Q_1$ is a finite set of \textit{arrows}, whose elements are written as $\alpha\colon s(\alpha)\rightarrow t(\alpha)$.
The \textit{path algebra} $K\Q$ is defined as the $K$-vector space with a basis consisting of all paths in $\Q$, that is, sequences of arrows $\omega=\alpha_s\punkte\alpha_1$ with $t(\alpha_{k})=s(\alpha_{k+1})$ for all $k\in\{1,\punkte,s-1\}$; formally included is a path $\epsilon_i$ of length zero for each $i\in \Q_0$ starting and ending in $i$. The multiplication  is defined as the  concatenation of paths $\omega= \alpha_s ... \alpha_1$ and $\omega' = \beta_t ... \beta_1$, that is,
\begin{center}
 $\omega\cdot\omega'=\left\{\begin{array}{ll}\alpha_s ... \alpha_1\beta_t ... \beta_1,&~\textrm{if}~t(\beta_t)=s(\alpha_1);\\
0,&~\textrm{otherwise.}\end{array}\right.$\end{center}

Let $\rad(K\Q)$ be the \textit{path ideal} of $K\Q$,  which is the (two-sided) ideal generated by all paths of positive lengths. An ideal $I\subseteq K\Q$ is called \textit{admissible} if there exists an integer $s$ with $\rad(K\Q)^s\subset I\subset\rad(K\Q)^2$. If this is the case for an ideal $I$, then the algebra $\A:=K\Q/I$ is finite-dimensional.

We denote by $\rep(K\Q)$ the abelian $K$-linear category of all representations of $\Q$ (which is equivalent to the category of $K\Q$-modules). In more detail, the objects are given as finite-dimensional \textit{($K$-)representations} of $\Q$ which, in more detail, are tuples \[((M_i)_{i\in \Q_0},(M_\alpha\colon M_i\rightarrow M_j)_{(\alpha\colon i\rightarrow j)\in \Q_1}),\] where the $M_i$ are $K$-vector spaces, and the $M_{\alpha}$ are $K$-linear maps. A \textit{morphism of representations} $M=((M_i)_{i\in \Q_0},(M_\alpha)_{\alpha\in \Q_1})$ and
 \mbox{$M'=((M'_i)_{i\in \Q_0},(M'_\alpha)_{\alpha\in \Q_1})$} consists of a tuple of $K$-linear maps $(f_i\colon M_i\rightarrow M'_i)_{i\in \Q_0}$, such that $f_jM_\alpha=M'_\alpha f_i$ for every arrow $\alpha\colon i\rightarrow j$ in $\Q_1$.
 
Let us denote by   $\rep(\A)$ the category of representations of $\Q$ bound by $I$: For a representation $M$ and a path $\omega$ in $\Q$ as above, we denote $M_\omega=M_{\alpha_s}\cdot\punkte\cdot M_{\alpha_1}$. A representation $M$ is called \textit{bound by $I$}, if $\sum_\omega\lambda_\omega M_\omega=0$ whenever $\sum_\omega\lambda_\omega\omega\in I$. The category $\rep(K\Q/I)$ is equivalent to the category of finite-dimensional $\A$-representations.

Let $M$ be an $\A$-representation, let $B_i\subseteq \epsilon_iM$ be a $K$-basis of $\epsilon_iM$ for every $i\in \Q_0$ and let $B$ be the disjoint union of these sets $B_i$. We define the \textit{coefficient quiver} $\Gamma(M):=\Gamma(M,B)$ of $M$ with respect to the basis $B$ to be the quiver with exactly one vertex for each element of
$B$, such that for each arrow $\alpha\in\Q_1$ and every element $b\in B_{s(\alpha)}$ we have \[M_{\alpha}(b)=\sum_{c\in B_{t(\alpha)}}\lambda_{b,c}^{\alpha}c \]
with $\lambda_{b,c}^{\alpha}\in K$. For each $\lambda_{b,c}^{\alpha}\neq 0$ we draw an arrow $b\rightarrow c$ with label $\alpha$. Thus, the quiver reflects the coefficients corresponding to the representation $M$ with respect to the chosen basis $B$.

Given a representation $M\in\rep(\A)$, its \textit{dimension vector} $\dimv M\in\mathbf{N}\Q_0$ is defined by $(\dimv M)_{i}=\dim_K M_i$ for $i\in \Q_0$. For a fixed dimension vector $\df\in\mathbf{N}\Q_0$, we denote by $\rep(\A,\df)$ the full subcategory of $\rep(\A)$ which consists of representations of dimension vector $\df$.

Let $M$ and $M'$ be two representations of $\A$. We denote by $\Hom_{\A}(M,M')$ the space of homomorphisms from $M$ to $M'$, by $\End_{\A}(M)$ the set of endomorphisms and by $\Aut_{\A}(M)$ the group of automorphisms of $M$ in $\rep(\A)$.

For certain finite-dimensional algebras a convenient tool for the classification of the indecomposable representations (up to isomorphism) and of their homomorphisms is the \textit{Auslander--Reiten quiver} $\Gamma(\A)$ of $\rep(\A)$. Its vertices $[M]$ are given by the isomorphism classes of indecomposable representations of $\rep(\A)$; the arrows between two such vertices $[M]$ and $[M']$ are parametrized by a basis of the space of so-called irreducible maps $f\colon M\rightarrow M'$.

By defining the affine space $\R_{\df}(K\Q):= \bigoplus_{\alpha\colon i\rightarrow j}\Hom_K(K^{d_i},K^{d_j})$, one realizes that its points $m$ naturally correspond to representations $M\in\rep(K\Q,\df)$ with $M_i=K^{d_i}$ for $i\in \Q_0$. 
 Via this correspondence, the set of such representations bound by $I$ corresponds to a closed subvariety $\R_{\df}(\A)\subset \R_{\df}(K\Q)$.
 
The algebraic group $\GL_{\df}=\prod_{i\in \Q_0}\GL_{d_i}$ acts on $\R_{\df}(K\Q)$ and on $\R_{\df}(\A)$ via base change, furthermore the $\GL_{\df}$-orbits $\Orb_M$ of this action are in bijection to the isomorphism classes of representations $M$ in $\rep(\A,\df)$.

The notion of symmetry for a finite quiver comes into the picture as follows: A \textit{symmetric quiver} is a pair $(\Q,\sigma)$ where $\Q$ is a finite quiver and $\sigma:\Q_0\cup \Q_1\rightarrow \Q_0\cup \Q_1$ is an involution, such that  $\sigma(\Q_0)= \Q_0$, $\sigma(\Q_1)= \Q_1$ and 
every arrow $\xymatrix@1{i\ar^\alpha[r]&j}$ is sent to the arrow $\xymatrix@1{\sigma(j)\ar^{\sigma(\alpha)}[r]&\sigma(i)}$. 

In this article, we represent the action of $\sigma$ by adding the symbol $\ast$. For example, 
$$
\xymatrix{1\ar^a[r]&2\ar^b[r]&3\ar^{b^\ast}[r]&2^\ast\ar^{a^\ast}[r]&1^\ast}
$$
is the symmetric quiver $(\Q,\sigma)$ with underlying quiver $\Q$ being equioriented of type $A_5$, such that $\sigma$ acts on $\Q$ by sending an elment $x\in \Q_0\cup \Q_1$ to $x^\ast$; the vertex $3$ is fixed by $\sigma$. 

A \textit{symmetric ($K$-)representation} of a symmetric quiver $(\Q,\sigma)$  is a representation $M=(\{M_p\}_{p\in \Q_0},\{M_\alpha\}_{\alpha\in \Q_1})$ in $\rep(K\Q)$ endowed with a  non--degenerate bilinear form \[\langle-,-\rangle:\bigoplus_{p\in \Q_0}M_p\times \bigoplus_{q\in \Q_0}M_q\rightarrow K,\] such that:
\begin{enumerate}
\item The equation \begin{equation}\label{Eq:ConditionBilinearForm}
\langle-,-\rangle|_{M_p\times M_q}=0
\end{equation}
holds true, unless $q=\sigma(p)$;
\item  The equation
\begin{equation}\label{Eq:ConditionSymmetricRep}
\langle M_\alpha(v),w\rangle+\langle v,M_{\sigma(\alpha)}(w)\rangle=0
\end{equation} holds true for every $v\in M_p$, $w\in M_{\sigma(q)}$ and for every arrow $\xymatrix{p\ar^\alpha[r]&q}\in\Q_1$.
\end{enumerate}

A representation $(M,\langle-,-\rangle)$ of a symmetric quiver $(\Q,\sigma)$ is called \emph{symplectic}, if the bilinear form is skew--symmetric and it is called \emph{orthogonal}, if the bilinear form is symmetric. 

Let $(\Q,\sigma)$ be a symmetric quiver and let $I$ be an ideal of $K\Q$, such that $\sigma\cdot I\subset I$. The involution $\sigma$ induces an involution on the algebra $\A:=K\Q/I$ and we can consider symplectic and  orthogonal  representations of the algebra $\A$: these are symplectic or orthogonal representations of $\A$  which are annihilated by the ideal $I$. 

We denote the categories of symmetric, symplectic and orthogonal representations by $\SRep(\A)$ and make sure that it will always be clear from the context which one is meant. The restriction to the full subcategory of representations of a fixed dimension vector $\df$ is denoted by $\SRep(\A,\df)$. Analogously to the non-symmetric case, we associate a variety $\SR_{\df}(\A)$ to this category; and denote $M_x\in \SRep(\A,\df)$ for $x\in\SR_{\df}(\A)$. 

Let $(M,\form)$ and $(M',\form')$ be two representations in $\SRep(\A)$.  Let us denote by $\Hom_{\SRep(\A)}(M,M')$ the space of homomorphisms from $M$ to $M'$ and by $\Aut_{\SRep(\A)}(M)$ the group of automorphisms of $M$ in $\SRep(\A)$. 

We have
$$
\Aut_{\SRep(K\Q)}(M)=\Aut_{K\Q}(M)\bigcap\prod_{p\neq\sigma(p)}\Sym(M_p\oplus M_{\sigma(p)})\times\prod_{p=\sigma(p)}\Sym(M_p).
$$

 Let $\End_{\SRep(\A)}(M)$ be the Lie algebra of $\Aut_{\SRep(\A)}(M)$. An element $A=\{A_p\}_{p\in \Q_0}$ of $\End_{\SRep(\A)}(M)$ is called a \emph{symmetric endomorphism} of $(M,\form)$; it is an element of $\End_{\A}(M)$ with the following extra conditions:
\begin{eqnarray}\nonumber
\textrm{If~$M$~ is~ symplectic:}&&\textrm{ $p\neq\sigma(p) \Rightarrow A_{\sigma(p)}=-\presuper{\mathfrak{T}}{{A}_p}$ }\\\label{Eq:SymEndoCond1}
\textrm{If~$M$~is~ orthogonal:}&&\textrm{ $p\neq\sigma(p) \Rightarrow A_{\sigma(p)}=\presuper{\mathfrak{T}}{{A}_p}$}\\\label{Eq:SymEndoCond2}
\textrm{Furthermore,} &&\textrm{$p=\sigma(p)\Rightarrow
\presuper{T\!\!\!}{A}_pF_p+F_p A_p=0.$}
\end{eqnarray}
\begin{remark}
Conditions \eqref{Eq:SymEndoCond1} follow by imposing the relation
$$
\left[\begin{array}{cc}\presuper{T\!\!\!}{A_p}&0\\0&\presuper{T\!\!\!}{A_{\sigma(p)}}\end{array}\right]\,\left[\begin{array}{cc}0&J\\\pm J&0\end{array}\right]+
\left[\begin{array}{cc}0&J\\\pm J&0\end{array}\right]\,\left[\begin{array}{cc}A_p&0\\0&A_{\sigma(p)}\end{array}\right]=0.
$$
Condition \eqref{Eq:SymEndoCond2} means that $A_p$ belongs to the Lie algebra of $\Sym(V_p)$.
\end{remark}
We hence have 
$$
\End_{\SRep(K\Q)}(M)=\End_{K\Q}(M)\bigcap\prod_{p\neq\sigma(p)}\fsym(V_p\oplus V_{\sigma(p)})\times\prod_{p=\sigma(p)}\fsym(V_p).
$$

 \section{B-orbits vs. isoclasses of symmetric representations}\label{sect:translation}
Let $G\in \{\SP_n, \Or_n\}$ where $n=2l$ in the symplectic case and $n\in\{2l, 2l+1\}$ in the orthogonal case for some integer $l\in\mathbf{N}$ and let $\fg$ be the corresponding symplectic or orthogonal Lie algebra.
Let $B$ be the standard Borel subgroup of $G$, that is, the subgroup of $G$ of upper-triangular matrices which is obtained by intersecting the Borel subgroup of $\GL_n$ with $G$. 

 We consider the algebraic variety $\N(2)$ of $2$--nilpotent elements of $\mathfrak{g}$
\[\N(2)=\N(2,G)=\{x\in\mathfrak{g}|\,x^2=0\}.\]
Then $B$ acts on $\N(2)$ via conjugation and our first aim in this article is to prove by means of symmetric quiver representations that the action admits only a finite number of orbits. We thereby specify an explicit parametrization of the orbits.

\subsection{Symmetric quiver setup}\label{ssect:quivSetup}

We define   $\A(l)$ to be the algebra given by the quiver
$$
Q_l:\;\xymatrix{
1\ar^{a_1}[r]&2\ar^{a_2}[r]&\cdots\ar^{a_{k-1}}[r]&k\ar^{a_k}[r]&\omega\ar^\alpha@(lu,ru)\ar^{a_k^\ast}[r]&k^\ast\ar^{a_{k-1}^\ast}[r]&\cdots\ar^{a_2^\ast}[r]&2^\ast\ar^{a_1^\ast}[r]&1^\ast
}
$$
with relations $\alpha^2=a_l^\ast a_l=0$. Notice that the $2l$ vertices of $Q_l$ are colored; the choice of the color will be clear in a few lines.

We consider the dimension vector \[\mathbf{d}_\bullet=(d_1,...,d_l,d_{\omega},d_{k^*},...,d_{1^*})=(1,2,\cdots,l-1,l,n,l,l-1,\cdots, 2,1)\] 
and the variety $\SR_{\mathbf{d}_\bullet}(\A(l))$. This variety is acted upon by the group 
\[\GL_{\sym}:=\GL(d_1)\times \GL(d_2)\times\cdots\times \GL(d_l)\times \Sym(n)\] where $\Sym(n)$ denotes either the symplectic or the orthogonal group on a vector space of dimension $n$. Inside the variety $\SR_{\mathbf{d}_\bullet}(\A(l))$ we consider the open subset $\SR_{\mathbf{d}_\bullet}(\A(l))^0$ corresponding to the full subcategory $\SRep(\A(l),\mathbf{d}_\bullet)^0$ of $\SRep(\A(l),\mathbf{d}_\bullet)$ of those representations whose linear maps associated with the arrows $a_i$ and $a_i^\ast$ have maximal rank. For an element $x\in\SR_{\mathbf{d}_\bullet}(\A(l))^0$, we have the natural notion of its $\GL_{\sym}$-stabilizer.  We denote the corresponding stabilizer of the representation $M_x$ in the category $\SRep(\A(l),\mathbf{d}_\bullet)$ by $\stab_{\GL_{\sym}}(M_x)$.

\begin{example}\label{ex:Borel}
Let us consider the quiver $\Q_2$
$$
\xymatrix{
1\ar^a[r]&2\ar^b[r]&3\ar^{b^\ast}[r]\ar^\alpha@(ul,ur)&2^\ast\ar^{a^\ast}[r]&1^\ast
}
$$
and the algebra $\A(2)=K\Q/(\alpha^2,b^\ast b)$.  Let us consider the $\A(2)$-representation $M_0$ given by
\begin{equation}\label{Ex:Flag4}
\xymatrix@R=3pt{
1\ar^a[r]&2\ar^b[r]&3&&\\
&2\ar^b[r]&3&&\\
&&3\ar^{b^\ast}[r]&2^\ast&\\
&&3\ar^{b^\ast}[r]&2^\ast\ar^{a^\ast}[r]&1^\ast\\
}
\end{equation}
and the $\A(2)$-representation $M_0'$ given by
\begin{equation}\label{Ex:Flag4'}
\xymatrix@R=3pt{
1\ar^a[r]&2\ar^b[r]&3&&\\
&2\ar^b[r]&3&&\\
&&3&&\\
&&3\ar^{b^\ast}[r]&2^\ast&\\
&&3\ar^{b^\ast}[r]&2^\ast\ar^{a^\ast}[r]&1^\ast\\
}
\end{equation}
In view of \eqref{Eq:ConditionSymmetricRep}, in order for $M$ to be symmetric, the arrows $a,b$ of $\Q_2$  must act by $1$ and the arrows $a^\ast,b^\ast$ must act as $-1$.

The symmetric structure of $M_0$ (that is, the choice of a non-degenerate bilinear form) is induced by the symmetric structure on the vector space at vertex $3$. In the symplectic case, this bilinear form is given by the matrix
$$
\left[\begin{array}{cccc}
0&0&0&1\\
0&0&1&0\\
0&-1&0&0\\
-1&0&0&0\\
\end{array}
\right];$$ 
and for the orthogonal case, it is defined by the anti--diagonal matrix with every entry on the anti--diagonal being $1$ (see Section \ref{sect:classicalLie}).

The symplectic space $\End_{\SRep(\A(2))}(M_0)$ has dimension $6$ and can be represented by the matrix
$$
\left[\begin{array}{cc|cc}
a&c&f&g\\
0&b&e&f\\\hline
0&0&-b&-c\\
0&0&0&-a\\
\end{array}
\right].$$ 
In orthogonal type, it is $4$--dimensional and represented by
$$
\left[\begin{array}{cc|cc}
a&c&f&0\\
0&b&0&-f\\\hline
0&0&-b&-c\\
0&0&0&-a\\
\end{array}
\right].$$ 
In a similar way, we proceed for the orthogonal group $\Or_{5}$ and look at the representation $M_0'$. 
Then, as above, the stabilizer is given by
$$
\left[\begin{array}{ccc|cc}
a&c&e&f&0\\
0&b&d&0&-f\\\hline
0&0&0& -d&-e\\
0&0&0& -b&-c\\
0&0&0&0& -a\\
\end{array}
\right]$$ 
We have hence found the well known fact that the stabilizer of the complete standard flag is the Borel subgroup in both the symplectic and orthogonal setup. The Borel subgroups therefore equal the intersection of $G$ with the standard Borel of upper-triangular matrices in $\GL_n$. 
\end{example}
Clearly, this example generalizes to larger $n$ in a straight forward manner.

 \subsection{Translation}
The translation from  $B$--orbits in $\N(2)$ to the representation theory of a symmetric quiver is based on a theorem on associated fibre bundles which we recall for the convenience of the reader. Its origin can be found in \cite{Se}.
\begin{theorem}\label{Thm:BRBundle}
Let $G$ be an algebraic group, let $X$ and $Y$ be $G$--varieties, and let $\pi:X\rightarrow Y$ be a $G$--equivariant morphism. 
Assume that $Y$ is a single $G$--orbit, $Y = Gy_0$. Define $H := \textrm{Stab}_G(y_0) = \{g\in G|\, g\cdot y_0 = y_0\}$ and $F:=\pi^{-1}(y_0)$. Then $X$ is isomorphic to the associated fibre bundle $G\times^H F$, and 
the embedding $\iota:F\rightarrow X$ induces a bijection between $H$--orbits in $F$ and $G$--orbits in $X$ preserving orbit closures.
\end{theorem}
\begin{corollary}\label{Cor:Stabilizer}
With the notation of Theorem \ref{Thm:BRBundle}, given a point $p\in F$, we have $\stab_H(p)=\stab_G(p)$
\end{corollary}
\begin{proof}
Since $H$ is  a subgroup of $G$, $\textrm{Stab}_H(p)\subseteq \textrm{Stab}_G(p)$; viceversa, since $H\cdot p=G\cdot p\cap F$, the reversed inclusion also holds. 
\end{proof}

 In view of Theorem \ref{Thm:BRBundle}, we can now prove the following key lemma, analogous to \cite[Lemma~3.1]{B2}.
\begin{lemma}\label{lem:Bijection}
There is a bijection between isoclasses of symplectic/orthogonal $\A(l)$--representations in $\SRep(\A(l), \mathbf{d}_\bullet)^0$ and symplectic/orthogonal $B$--orbits in $\N(2)$. This bijection respects orbit closure relations and dimensions of stabilizers. 
\end{lemma}
\begin{proof}
Let $\widetilde{\Q_l}$ be the quiver obtained from $\Q_l$ by removing the loop $\alpha$ and let $\widetilde{\A(l)}$ be the corresponding symmetric algebra  (also remove the relation $\alpha^2$). By defining $\SR_{\mathbf{d}_\bullet}(\widetilde{\A(l)})^0$ analougously to $\SR_{\mathbf{d}_\bullet}(\A(l))^0$, we see that this variety is acted upon transitively by $\GL_{\sym}$ and we denote the representation which is given by the complete standard flag by $M_0$; this is a generating point.  The embedding $\widetilde{\A(l)}\subset\A(l)$ induces a  $\GL_{\sym}$-equivariant projection
$$
\xymatrix{
\pi:\SR_{\mathbf{d}_\bullet}(\A(l))^0\ar@{->>}[r]&\SR_{\mathbf{d}_\bullet}(\widetilde{\A(l)})^0}
$$
which is given by forgetting the linear map associated with the loop $\alpha$. The fiber of $\pi$ equals the variety $\N(2)$.

As we have seen before, the stabilizer of the symplectic/orthogonal representation $M_0$ is isomorphic to the Borel subgroup $B$ of the symplectic/orthogonal group. Thus,  Theorem \ref{Thm:BRBundle} proves the claim.
\end{proof}
We are hence left to classify the isomorphism classes of symplectic/orthogonal representations of $\A(l)$ of dimension vector $\mathbf{d}_\bullet$ with maximal rank maps, which in view of Krull--Remak--Schmidt's theorem is analogous to classifying the unique decompositions of  elements of $\SRep(\A(l), \mathbf{d}_\bullet)^0$ into indecomposable symplectic/orthogonal representations (up to symmetric isomorphism). Let $M$ and $M'$ be two points of $\SRep(\A(l), \mathbf{d}_\bullet)^0$ which are contained in different orbits. Since $\pi(M)=\pi(M')$ under the morphism of the proof of Lemma \ref{lem:Bijection}, the only difference beetween them is given by the action of the loop $\alpha$. This means that the only part of the coefficient quivers of $M$ and $M'$ which differs is the subquiver which represents the loop $\alpha$.

\section[Representation theory of A(l)]{Representation theory of $\A(l)$}\label{sect:repAl}
In this section, we look at the (symmetric) representation theory of the algebra $\A(l)$ corresponding to the symmetric quiver $\Q_l$. With these considerations, we are able to prove explicit parametrizations of the Borel-orbits in $\N(2)$ in Section \ref{sect:paramOrbits}.
\subsection[Indecomposable symmetric A(l)--modules]{Indecomposable symmetric $\A(l)$--modules}\label{Subsec:Indecomposables}
The following proposition follows from \cite[Section~3]{BuRi} by noticing that there are no band modules. 
\begin{proposition}
The algebra $\A(l)$ is a string algebra of finite representation type. In particular, the indecomposable $\A(l)$--modules are string modules and their isoclasses are parametrized by words with letters in the arrows of $\Q_l$ and their inverses, avoiding relations. 
\end{proposition}

Let us give names to the indecomposable $\A(l)$--modules (where $l+1:=\omega$). 
\begin{itemize}
\item[$M_{ij}$:] For $1\leq i\leq j\leq l+1$, we denote by $M_{ij}$ the string module associated with the word $a_i\cdots a_{j-1}$, i.e. it is the indecomposable module supported on vertices $i,i+1,\cdots, j$; its coefficient quiver is given by
$$
\xymatrix{
i\ar[r]&i+1\ar[r]&\cdots\ar[r]&j-1\ar[r]&j
}
$$
\item[$M_{ij}^\ast$:] For $1\leq i\leq j\leq l+1$, we denote by $M_{ij}^\ast$ the string module associated with the word $a_{j-1}^{\ast}\cdots a_i^\ast$, i.e. it is the indecomposable module supported on vertices $j^\ast,(j-1)^\ast,\cdots, i^\ast$; its coefficient quiver is given by
$$
\xymatrix{
j^\ast\ar[r]&(j-1)^\ast\ar[r]&\cdots\ar[r]&(i+1)^\ast\ar[r]&i^\ast
}
$$

\item[$D_{ij}^+$:] For $1\leq i\leq j\leq l+1$, we denote by $D_{ij}^+$ the indecomposable associated with the word $a_ia_{i+1}\cdots a_l\alpha^- a_l^-a_{l-1}^-\cdots a_j^-$; its coefficient quiver has the following form
$$
\xymatrix{
i\ar[r]&i+1\ar[r]&\cdots\ar[r]&j\ar[r]&j+1\ar[r]&\cdots\ar[r]&l\ar[r]&\omega\\
&&&j\ar[r]&j+1\ar[r]&\cdots\ar[r]&l\ar[r]&\omega\ar[u]
}
$$
\item[$D_{ij}^-$:] For $1\leq i<j\leq l+1$ we denote by $D_{ij}^-$ the indecomposable associated with the word $a_ia_{i+1}\cdots a_l\alpha a_l^-a_{l-1}^-\cdots a_j^-$; its coefficient quiver is given by
$$
\xymatrix{
i\ar[r]&i+1\ar[r]&\cdots\ar[r]&j\ar[r]&j+1\ar[r]&\cdots\ar[r]&l\ar[r]&\omega\ar[d]\\
&&&j\ar[r]&j+1\ar[r]&\cdots\ar[r]&l\ar[r]&\omega
}
$$

\item[$C_{ij}^+$:] For $1\leq i\leq j\leq l+1$ we denote by $C_{ij}^+$ the indecomposable associated with the word $(a_j^\ast)^- (a_{j-1}^\ast)^-\cdots (a_l^\ast)^-\alpha^- a_l^\ast a_{l-1}^\ast\cdots a_i^\ast$; its coefficient quiver is given by
$$
\xymatrix{
\omega\ar[d]\ar[r]&l^\ast\ar[r]&\cdots\ar[r]&j^\ast\ar[r]&(j-1)^\ast\ar[r]&\cdots\ar[r]&(i+1)^\ast\ar[r]&i^\ast\\
\omega\ar[r]&l^\ast\ar[r]&\cdots\ar[r]&j^\ast&&&&
}
$$
\item[$C_{ij}^-$:] For $1\leq i< j\leq l+1$ we denote by $C_{ij}^-$  the indecomposable module associated with the word $(a_j^\ast)^- (a_{j-1}^\ast)^-\cdots (a_l^\ast)^-\alpha\, a_l^\ast a_{l-1}^\ast\cdots a_i^\ast$; its coefficient quiver is given by
$$
\xymatrix{
\omega\ar[r]&l^\ast\ar[r]&\cdots\ar[r]&j^\ast\ar[r]&(j-1)^\ast\ar[r]&\cdots\ar[r]&(i+1)^\ast\ar[r]&i^\ast\\
\omega\ar[u]\ar[r]&l^\ast\ar[r]&\cdots\ar[r]&j^\ast&&&&
}$$

\item[$Z_{ij}^+$:] For $1\leq i, j\leq l$ we denote by $Z_{ij}^+$ the indecomposable associated with the word $a_ia_{i+1}\cdots\alpha_l\alpha^- a_l^\ast\cdots a_j^\ast$; its coefficient quiver is given by
$$
\xymatrix{
i\ar[r]&i+1\ar[r]&\cdots\ar[r]&l\ar[r]&\omega&&&\\
&&&&\omega\ar[u]\ar[r]&l^\ast\ar[r]&\cdots\ar[r]&j^\ast
}
$$
\item[$Z_{ij}^-$:] For $1\leq i, j\leq l$ we denote by $Z_{ij}^-$ the indecomposable associated with the word $a_ia_{i+1}\cdots a_l\alpha a_l^\ast\cdots a_j^\ast$; its coefficient quiver is given by
$$
\xymatrix{
i\ar[r]&i+1\ar[r]&\cdots\ar[r]&l\ar[r]&\omega\ar[d]&&&\\
&&&&\omega\ar[r]&l^\ast\ar[r]&\cdots\ar[r]&j^\ast
}
$$
\end{itemize}
\begin{remark}\label{Rem:Scalars}
All the modules above are non--isomorphic to each other, apart from $D_{l+1,l+1}^+\simeq C_{l+1,l+1}^+$ and $M_{l+1,l+1}\simeq M_{l+1,l+1}^\ast$.
\end{remark}

We consider the involution $\sigma$ of $\Q_l$, which sends every vertex $i$ to $i^\ast$, every arrow $a$ to $a^\ast$ and which fixes $\omega$ and $\alpha$ (here we use the convention that $(-)^{\ast\ast}=(-)$). Then $(\Q_l,\sigma)$ is a symmetric quiver and we can consider symmetric representations of $\A(l)$. The involution $\sigma$ induces a duality on the category of representations of $\A(l)$ that we denote by $\nabla$ (as in \cite{DW}). 
\begin{convention}
Given an indecomposable $\A(l)$--module $M$, we need to choose carefully the linear maps. Since we often work with its coefficient quiver, we fix one and for all a convention about these:\\
The arrows  of the coefficient quiver of $M$ colored with $a_1,\cdots, a_l$ act as $1$, while  the arrows   colored with $a_l^\ast,\cdots, a_1^\ast$ act as $-1$. 

Every pair of two arrows $\xymatrix{\omega_i\ar^{\alpha_1}[r]&\omega_j}$ and $\xymatrix{\omega_{j^\ast}\ar^{\alpha_2}[r]&\omega_{i^\ast}}$ colored with $\omega$ (if they exist) has to satisfy the following conditions:
\begin{itemize}
\item For $V$ to be orthogonal, $\alpha_1$ acts as $1$ and $\alpha_2$ as $-1$.
\item  For $V$ to be symplectic, if $1\leq i,j\leq l$ or $1\leq i^\ast,j^\ast\leq l$, then $\alpha_1$ acts as $1$ and $\alpha_2$ as $-1$, otherwise $\alpha_1$ and $\alpha_2$ both act as $1$.
\end{itemize}
\end{convention}

\begin{proposition}\label{Prop:Nabla}
With the above notation, we have: 
$\nabla M_{i,j}\simeq M_{i,j}^\ast$, $\nabla D_{i,j}^+\simeq C_{i,j}^+$, $\nabla D_{i,j}^-\simeq C_{i,j}^-$, $\nabla Z_{i,j}^+\simeq Z_{j,i}^+$, $\nabla Z_{i,j}^-\simeq Z_{j,i}^-$. In particular,  $\nabla M_{l+1,l+1}\simeq M_{l+1,l+1}$ and $\nabla D_{l+1,l+1}^+\simeq D_{l+1,l+1}^+$
\end{proposition}
\begin{proof}
Let M be an indecomposable module as listed above. The coefficient quiver of the dual $\nabla M$ of $M$ is obtained from the coefficient quiver of $M$ by reversing all the arrows, changing their sign and then making a reflection through the middle vertex $\omega=l+1$. 
\end{proof}
Thus, we obtain the following classification lemma.

\begin{lemma}
The symplectic indecomposable representations of $\A(l)$ are $Z_{ii}^\pm$, $M_{ij}\oplus M_{ij}^\ast$, $D_{ij}^\pm\oplus C_{ij}^\pm$ (for $(i,j)\neq (l+1,l+1)$), $D_{l+1,l+1}^+$ and $Z_{ij}^\pm\oplus Z_{ji}^\pm$ (for $i\neq j$).\\[1ex]
The orthogonal indecomposable representations of $\A(l)$ are $M_{ij}\oplus M_{ij}^\ast$, $D_{ij}^\pm\oplus C_{ij}^\pm$, $Z_{ij}^\pm\oplus Z_{ji}^\pm$ and $M_{l+1,l+1}$.
\end{lemma} 

In particular, there is only one indecomposable $\A(l)$--modules which can be endowed with an orthogonal structure. 

\begin{remark}
The reason why an indecomposable $\A(l)$--module with symmetric dimension vector cannot be orthogonal, except for the case that it is one--dimensional, is the following: let $M$ be such a (at least two--dimensional) module and let $M_\alpha$ be the linear map associated with the loop $\alpha$. Such a map is a $2$--nilpotent endomorphism of an orthogonal two--dimensional vector space. In order for $M$ to be orthogonal, $M_\alpha$ must  lie in the Lie algebra $\fo_2$ of $\Or_2$ and hence it must be zero, contradicting the fact that $M$ is indecomposable.
\end{remark}

For example, the following representation:
$$
\xymatrix@R=6pt{
1\ar^1[r]&2\ar^1[r]&3\ar^1[r]&4\ar^1[r]&\omega\ar@/^2.5pc/^1[dddd]&&\\
&2\ar^1[r]&3\ar^1[r]&4\ar^1[r]&\omega\ar@/^1pc/^1[d]&&\\
&&3\ar^1[r]&4\ar^1[r]&\omega&&\\
&&&4\ar^1[r]&\omega\ar@/_3pc/_b[dddd]&&\\
&&&&\omega\ar^{-1}[r]&4^\ast&&&\\
&&&&\omega\ar@/_1pc/_{-1}[d]\ar^{-1}[r]&4^\ast\ar^{-1}[r]&3^\ast&&\\
&&&&\omega\ar^{-1}[r]&4^\ast\ar^{-1}[r]&3^\ast\ar^{-1}[r]&2^\ast&\\
&&&&\omega\ar^{-1}[r]&4^\ast\ar^{-1}[r]&3^\ast\ar^{-1}[r]&2^\ast\ar^{-1}[r]&1^\ast\\
}
$$
is symplectic if $b=1$ and orthogonal if $b=-1$. 

\subsection[Auslander--Reiten quiver of A(l)]{Auslander--Reiten quiver of $\A(l)$}\label{ssect:ARQ}

The algebra $\A(l)$ is a string algebra of finite representation--type, that is, it does only admit a finite number of isomorphism classes of indecomposable representations. Its Auslander--Reiten quiver can be obtained in several ways. We prefer to follow the treatment of Butler--Ringel \cite{BoRe} and get the following result. 

\begin{proposition}
The following are the Auslander--Reiten  sequences of $\A(l)$:

\begin{enumerate}
\item Auslander--Reiten sequences starting with $M_{ij}$:

$
\xymatrix{
0\ar[r]&M_{1, \omega}\ar[r]&Z_{1,1}^+\ar[r]&M_{1, \omega}^\ast\ar[r]&0
}
$,

$
\xymatrix{
0\ar[r]&M_{i, \omega}\ar[r]&M_{i-1, \omega}\oplus Z_{i, 1}^+\ar[r]&Z_{i-1, 1}^+\ar[r]&0
}
$, if $i>1$,

$
\xymatrix{
0\ar[r]&M_{i, j}\ar[r]&M_{i, j-1}\oplus M_{i-1, j}\ar[r]&M_{i-1, j-1}\ar[r]&0
}
$, if $i>1$ and $j\leq l$,

$
\xymatrix{
0\ar[r]&M_{i,i}=S_i\ar[r]&M_{i-1, i}\ar[r]&M_{i-1, i-1}\ar[r]&0
}
$, if $i=j>1$.

\item  Auslander--Reiten sequences starting with $M_{ij}^\ast$:

$
\xymatrix{
0\ar[r]&M_{1,j}^\ast\ar[r]&M_{2, j}^\ast\oplus M_{1, j+1}^\ast\ar[r]&M_{2, j+1}^\ast\ar[r]&0
}
$, if $j\leq l-1$,

$
\xymatrix{
0\ar[r]&M_{1,l}^\ast\ar[r]&M_{2,l}^\ast\oplus P_\omega\ar[r]&C_{1, l}^-\ar[r]&0
}$, 

$
\xymatrix{
0\ar[r]&M_{1,\omega}^\ast\ar[r]&M_{2, \omega}^\ast\oplus C_{1,1}^+\ar[r]&C_{1,2}^-\ar[r]&0
}
$, 

$
\xymatrix{
0\ar[r]&M_{i,i}^\ast=S_{i^\ast}\ar[r]&M_{i,i+1}^\ast\ar[r]&M_{i+1,i+1}^\ast\ar[r]&0
}
$, if $i<l$,

$\xymatrix{
0\ar[r]&M_{l,l}^\ast=S_{l^\ast}\ar[r]&C_{1,l}^-\ar[r]&C_{1,\omega}^-\ar[r]&0
}
$,

$
\xymatrix{
0\ar[r]&S_\omega\ar[r]&M_{l,\omega}\oplus C_{1,\omega}^+\ar[r]&Z_{l,1}^+\ar[r]&0
}
$,

$
\xymatrix{
0\ar[r]&M_{i,\omega}^\ast\ar[r]&M_{i+1,\omega}^\ast\oplus C_{1,i}^+\ar[r]&C_{1,i+1}^+\ar[r]&0
}
$, if $1<i<l$,

$
\xymatrix{
0\ar[r]&M_{l,\omega}^\ast\ar[r]&S_\omega\oplus C_{1,l}^+\ar[r]&C_{1,\omega}^+\ar[r]&0
}
$,

$
\xymatrix{
0\ar[r]&M_{i,j}^\ast\ar[r]&M_{i+1,j}^\ast\oplus M_{i,j+1}^\ast\ar[r]&M_{i+1,j+1}^\ast\ar[r]&0
}
$, if $i>1$ and $j<l$.

\item Auslander--Reiten sequences starting with $D_{ij}^+$:

$
\xymatrix{
0\ar[r]&D_{1,\omega}^+\ar[r]&D_{1,l}^+\oplus S_\omega\ar[r]&M_{l,\omega}\ar[r]&0
}
$,

$
\xymatrix{
0\ar[r]&D_{i,\omega}^+\ar[r]&D_{i-1,\omega}^+\oplus D_{i,l}^+\ar[r]&D_{i-1,l}^+\ar[r]&0
}
$, if $1<i\leq l$,

$
\xymatrix{
0\ar[r]&D_{1,j}^+\ar[r]&D_{1,j-1}^+\oplus M_{j,\omega}\ar[r]&M_{j-1,\omega}\ar[r]&0
}
$, if $1<j\leq l$,

$
\xymatrix{
0\ar[r]&D_{i,j}^+\ar[r]&D_{i-1,j}^+\oplus D_{i,j-1}^+\ar[r]&D_{i-1,j-1}^+\ar[r]&0
}
$, if $1<i\leq j\leq l$.

\item Auslander--Reiten sequences starting with $D_{ij}^-$: 

$
\xymatrix{
0\ar[r]&D_{1,\omega}^-\ar[r]&D_{1,l}^-\ar[r]&M_{l,l}=S_l\ar[r]&0
}
$,

$
\xymatrix{
0\ar[r]&D_{i,\omega}^-\ar[r]&D_{i-1,\omega}^-\oplus D_{i,l}^+\ar[r]&D_{i-1,l}^-\ar[r]&0
}
$, if $1<i\leq l$,

$
\xymatrix{
0\ar[r]&D_{i,i}^+\ar[r]&D_{i-1,i}^-\oplus D_{i-1,i}^+\ar[r]&D_{i-1,i-1}^+\ar[r]&0
}
$, if $1<i\leq\omega$

$
\xymatrix{
0\ar[r]&D_{1,j}^-\ar[r]&D_{1,j-1}^-\oplus M_{j,l}\ar[r]&M_{j-1,l}\ar[r]&0
}
$, if $1<j\leq l$,

$
\xymatrix{
0\ar[r]&D_{i,j}^-\ar[r]&D_{i-1,j}^-\oplus D_{i,j-1}^-\ar[r]&D_{i-1,j-1}^-\ar[r]&0
}
$, if $1<i< j\leq l$.

\item Auslander--Reiten sequences starting with $C_{ij}^+$:
 
$
\xymatrix{
0\ar[r]&C_{i,\omega}^+\ar[r]&C_{i+1,\omega}^+\oplus Z_{l,i}^+\ar[r]&Z_{l,i+1}^+\ar[r]&0
}
$, if  $1\leq i< l$, 

$
\xymatrix{
0\ar[r]&C_{l,\omega}^+\ar[r]&C_{\omega,\omega}^+\oplus Z_{l,l}^+\ar[r]&D_{l,\omega}^+\ar[r]&0
}
$,

$
\xymatrix{
0\ar[r]&C_{i,j}^+\ar[r]&C_{i+1,j}^+\oplus C_{i,j+1}^+\ar[r]&C_{i+1,j+1}^+\ar[r]&0
}
$, if $1<i\leq j\leq l$.

\item Auslander--Reiten sequences starting with $C_{ij}^-$: 

$
\xymatrix{
0\ar[r]&C_{1,\omega}^-\ar[r]&Z_{l,1}^-=P_l\oplus C_{2,\omega}^-\ar[r]&Z_{l,2}^-\ar[r]&0
}
$,

$
\xymatrix{
0\ar[r]&C_{i,\omega}^-\ar[r]&Z_{l,i}^-\oplus C_{i+1,\omega}^-\ar[r]&Z_{l,i+1}^-\ar[r]&0
}
$, if $1\leq i\leq l$,

$
\xymatrix{
0\ar[r]&C_{1,1}^-=P_\omega\ar[r]&C_{1,2}^-\oplus C_{1,2}^+\ar[r]&C_{2,2}^+\ar[r]&0
}
$,

$
\xymatrix{
0\ar[r]&C_{i,j}^-\ar[r]&C_{i+1,j}^-\oplus C_{i,j+1}^-\ar[r]&C_{i+1,j+1}^-\ar[r]&0
}
$, if $1<i\leq j\leq n$.

\item Auslander--Reiten sequences starting with $Z_{ij}^+$ (note that $Z_{i, \omega}^+=D_{i,\omega}^+$):

$
\xymatrix{
0\ar[r]&Z_{1,j}^+\ar[r]&Z_{1,j+1}^+\oplus M_{j,\omega}^\ast\ar[r]&M_{j+1,\omega}^\ast\ar[r]&0
}
$,
 
$
\xymatrix{
0\ar[r]&Z_{i,j}^+\ar[r]&Z_{i,j+1}^+\oplus Z_{i-1,j}^+\ar[r]&Z_{i-1,j+1}^+\ar[r]&0
}
$, if $i>1$. 

\item Auslander--Reiten sequences starting with $Z_{ij}^-$ (note that  $Z_{1j}^-=I_{j^\ast}$): 

$
\xymatrix{
0\ar[r]&Z_{i,1}^-=P_i\ar[r]&Z_{i-1,1}^-\oplus Z_{i,2}^-\ar[r]&Z_{i-1,2}^-\ar[r]&0
}
$, if $i>1$,

$
\xymatrix{
0\ar[r]&Z_{i,j}^-\ar[r]&Z_{i-1,j}^-\oplus Z_{i,j+1}^-\ar[r]&Z_{i-1,j+1}^-\ar[r]&0
}
$, if $1<i,j\leq l$,

\end{enumerate}
\end{proposition}
The resulting Auslander--Reiten quiver of $\A(l)$ has the shape of a "christmas tree"; its bottom part consists of pre--projective modules and its top consists of $l+1$ periodic $\tau$--orbits. The duality $\nabla$ acts as a reflection through the vertical line formed by the self--dual $\A(l)$--modules $Z_{ii}^\pm$ and $D^+_{l+1,l+1}$. Figure~\ref{Fig:AR-quiver A3} shows the Auslander--Reiten quiver of $\A(3)$.

\begin{figure}[ht]
\begin{center}
\includegraphics[trim=130 100 10 150,clip,width=220pt]{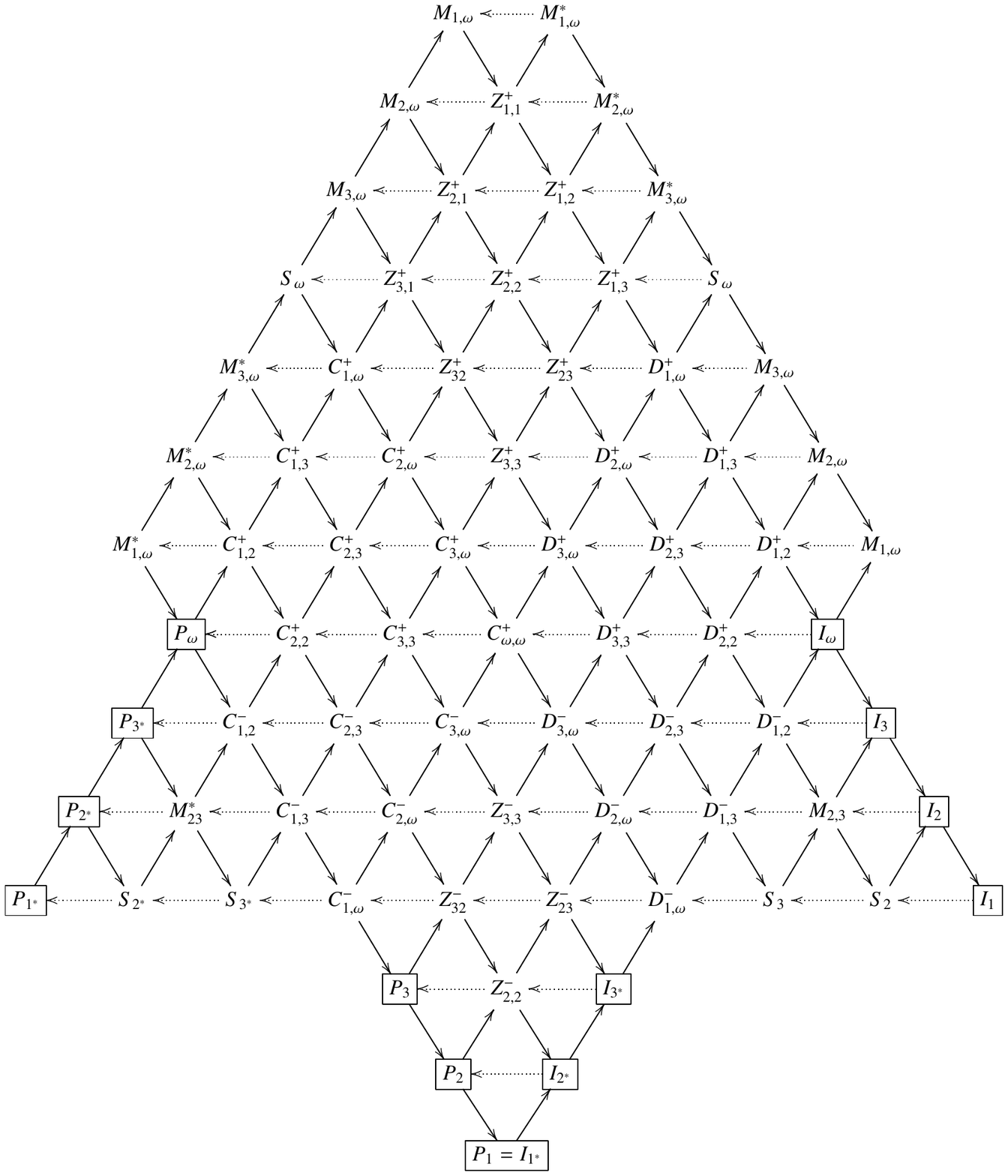}
\end{center}\caption{Auslander--Reiten quiver of $\A(3)$}\label{Fig:AR-quiver A3}
\end{figure}

\section{Parametrization of orbits}\label{sect:paramOrbits}
It is known by Panyushev \cite{Pan2} that $B$ acts finitely on the variety $\N(2)$. We aim to prove explicit parametrizations of the  orbits by means of symmetric representations and, thus, in a very combinatorial way. This way, we hope to be able to calculate e.g. degenerations in a follow-up article by means of the used representation-theoretic methods. We begin by discussing symplectic orbits in Subsection \ref{ssect:sympl_B} and deduce orthogonal orbits in Subsection \ref{ssect:orth_B}. In each type, we generalize the results to parabolic orbits in Section \ref{sect:ParGeneral}.

\subsection{Orbits in type C}\label{ssect:sympl_B}

Let $G=\SP_n$, where $n=2l$ for some integer $l$.
 We denote by $B$ the standard Borel subgroup of $G$ and  consider the algebra $\A(l)$ and its symmetric representations as discussed in \ref{sect:repAl}. Due to Lemma \ref{lem:Bijection}, we are interested in symplectic representations of dimension vector $\mathbf{d}_\bullet=(1,2,...,l,2l,l,...,2,1)$.
 
Let us begin with an example.

\begin{example}
Figure \ref{fig:exampleB} shows the complete list of isomorphism classes of symplectic representations 
in $\SRep(\A(2), \mathbf{d}_\bullet)^0$, where $n=4=2l$ and $B$ is the Borel subgroup. In more detail, the indecomposables are displayed by their coefficient quiver and then interpreted combinatorially by graphs on two vertices.

\begin{figure}[ht]
\begin{center}\includegraphics[trim=130 175 30 145,clip,width=220pt]{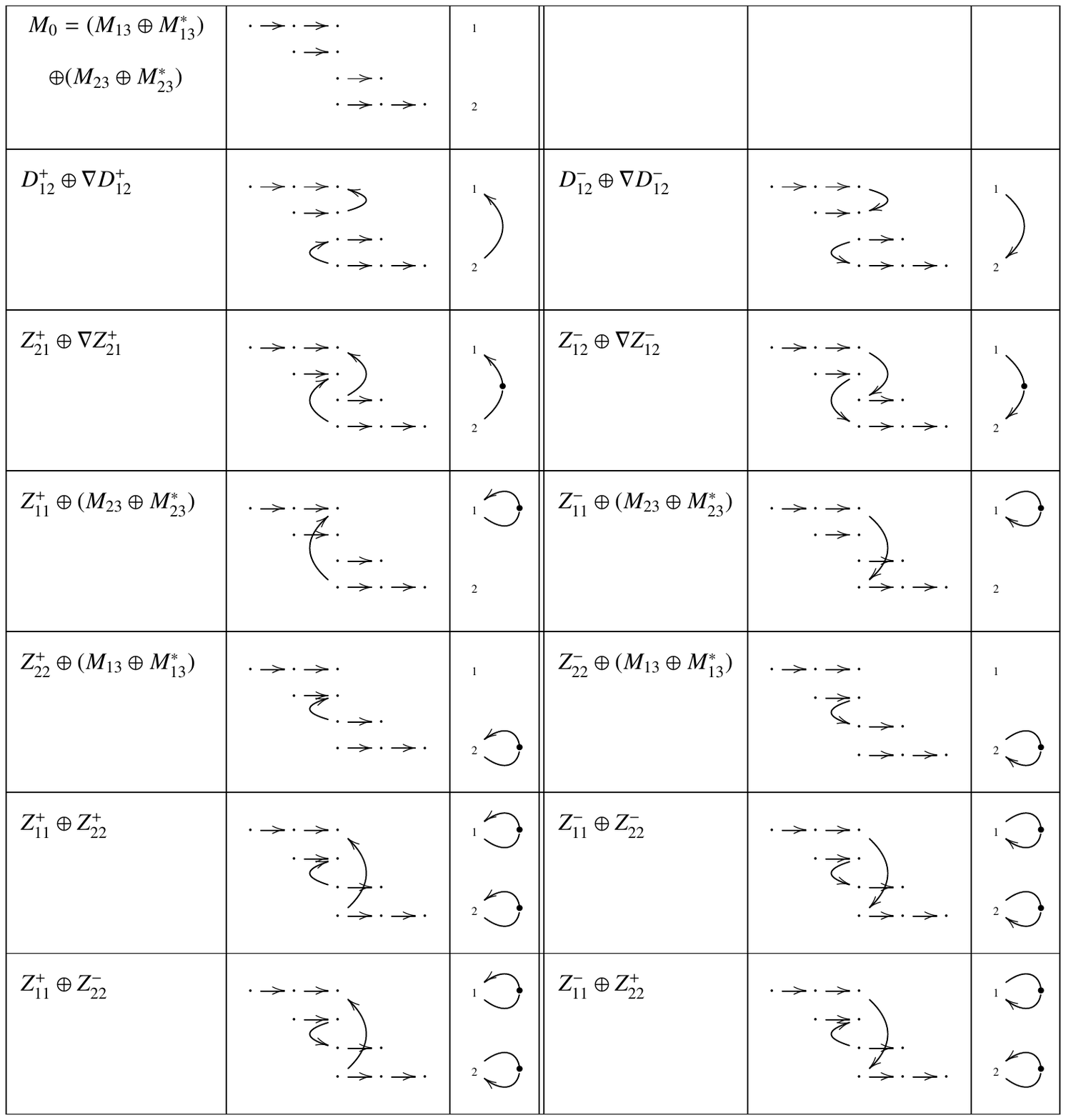}\end{center}\caption{Isomorphism classes of symplectic representations 
in $\SRep(\A(2), \mathbf{d}_\bullet)^0$ and their combinatorial interpretation}\label{fig:exampleB}
\end{figure}

\end{example}

This observation leads us to the following definition. 

\begin{definition}\label{Def:SOLP}
A \emph{symplectic oriented link pattern} (solp for short) of size $l$ consists of a  set of $l$ vertices $1,2,\cdots l$ together with a collection of oriented arrows between these vertices and a collection of oriented arrows with dots between these vertices, such that
\begin{itemize}
\item[(B)]\label{B} every vertex is touched by at most one arrow;
\item[(SpOr)]\label{SpOr} there are no loops of arrows without dots (that is, no arrows without dots from a vertex to itself).
\end{itemize}
We denote by $\Solp_{l}$ the set of solps of size $l$.
\end{definition}

The definition of symplectic link patterns leads to a combinatorial parametrization of symplectic Borel-orbits in $\N(2)$.

\begin{theorem}\label{Thm:symp_param_B}
The $B$--orbits in the variety $\N(2)\subseteq \fsp_n$ are in bijection with the set $\Solp_l$ of solps  of size $l$. 
\end{theorem}
\begin{proof}

By Krull--Remak--Schmidt, there is a quite obvious bijection between the set of solps of size $l$ and the set of symplectic representations  in $\SRep(\A(l),\mathbf{d}_\bullet)^0$ up to isomorphism which maps an isomorphism class $[M]$ of a symplectic representation $M$ to the subquiver of the coefficient quiver of $M$ induced by $M_\alpha$. Since this subquiver is completely determined by the vertices $1,...l$ and the arrows  By Lemma \ref{lem:Bijection}, the claim follows.

Let us state in detail, how the bijection works and how we can read of the representative $2$-nilpotent matrix of a particular solp. For this,  the following table gives a recipe; let $i<j$ and translate $i^\ast$ and $j^\ast$ via ($k^\ast=n-k+1$):

\[\begin{tabular}{|c|c|c|}
\hline
(Part of) solp &  Indecomposable & (Part of) Matrix\\\hline
\xymatrix@C=4pt{
\textrm{\tiny{i}}} & $M_{i,l+1}\oplus M_{i,l+1}^\ast$ & $0$ \\\hline
&&\\[-1ex]
\xymatrix{\textrm{\tiny{i}}&
\textrm{\tiny{j}}\ar@/_1pc/[l]}&$D^+_{i,j}\oplus C^+_{i,j}$& $E_{i,j}-E_{j^\ast,i^\ast}$\\\hline
&&\\[-1ex]
\xymatrix{\textrm{\tiny{i}}\ar@/^1pc/[r]&
\textrm{\tiny{j}}}&$D^-_{i,j}\oplus C^-_{i,j}$& $E_{j,i}-E_{i^\ast,j^\ast}$\\\hline
&&\\[-2ex]
\xymatrix{\textrm{\tiny{i}}&
\textrm{\tiny{j}}\ar@/_1pc/[l]|{\bullet}}&$Z^+_{i,j}\oplus Z^+_{j,i}$&$E_{i,j^\ast}+E_{j,i^\ast}$ \\\hline
&&\\[-2ex]
\xymatrix{\textrm{\tiny{i}}\ar@/^1pc/[r]|{\bullet}&
\textrm{\tiny{j}}}&$Z^-_{i,j}\oplus Z^-_{j,i}$& $E_{j^\ast,i}+E_{i^\ast,j}$\\\hline
&&\\[-2ex]
\xymatrix@C=4pt{
\textrm{\tiny{i}}\ar@(ur,ul)|{\bullet}}&$Z^+_{i,i}$&$E_{i,i^\ast}$ \\\hline
&&\\[-2ex]
\xymatrix@C=4pt{
\textrm{\tiny{i}}\ar@(ul,ur)|{\bullet}}&$Z^-_{i,i}$& $E_{i^\ast,i}$ \\\hline
\end{tabular} \]

\end{proof}
Since we have a combinatorial description, we can count the number of orbits.

\begin{proposition}
Let $s_{l}$ be the cardinality of~ $\Solp_{l}$. Then the sequence $\{s_{l}\}$ is determined by
\begin{itemize}
\item $s_0=1$,
\item $s_1=3$,
\item $s_{l}=3s_{l-1}+4(l-1)s_{l-2}$.
\end{itemize} 
\end{proposition}
\begin{proof}
We divide the set $\Solp_{l}$ into the subset of symmetric link patterns where vertex $1$ is not touched by any arrow and its complement. 
\end{proof}
The sequence $1,3,13, 63, 345,2043,...$ of numbers of slps is classified  in OEIS as A202837 \cite{OEIS_slp}. 
\begin{remark}
For $\GL_n$, the oriented link patterns considered in \cite{BoRe} only have to satisfy condition $(B)$ that is, the $2$--nilpotency conditions. We hence see that solps  are special oriented link patterns as defined in \cite{BoRe}. This is not surprising; indeed the following fact is known by \cite{MWZ}: if two symplectic elements are conjugate under the Borel of $\GL_n$, then they are conjugate under the Borel of $\SP_n$, as well. 
\end{remark}

We finish the symplectic classification by describing all $B$-orbits for $\SP(4)$ in detail.

\begin{example}
The $B$-orbits in $\N(2)\subseteq \mathfrak{sp}_n$ are classified by the collection of solps of size $2$. These are explicitly listed in the following table and the corresponding representative matrices in $\N(2)$ are displayed. This completes the example which we already considered in Figure \ref{fig:exampleB}.

$$
\tiny
\begin{tabular}{|c|c|c|c|c|}
\hline&&&&\\[-1ex]
\xymatrix{\textrm{\tiny{1}}\ar@/^1pc/[r]&
\textrm{\tiny{2}}}
&\xymatrix{\textrm{\tiny{1}}&
\textrm{\tiny{2}}\ar@/_1pc/[l]}
&\xymatrix{\textrm{\tiny{1}}\ar@/^1pc/[r]|{\bullet}&
\textrm{\tiny{2}}}
&\xymatrix{\textrm{\tiny{1}}&
\textrm{\tiny{2}}\ar@/_1pc/[l]|{\bullet}}&\xymatrix{\textrm{\tiny{1}}&
\textrm{\tiny{2}}}
\\\hline
&&&&\\[-1ex]
$\left(\begin{array}{cccc}
0&0&0&0\\
1&0&0&0\\
0&0&0&0\\
0&0&-1&0\end{array}\right)$ 
& $\left(\begin{array}{cccc}
0&1&0&0\\
0&0&0&0\\
0&0&0&-1\\
0&0&0&0\end{array}\right)$ 
&$\left(\begin{array}{cccc}
0&0&0&0\\
0&0&0&0\\
1&0&0&0\\
0&1&0&0\end{array}\right)$ 
&$\left(\begin{array}{cccc}
0&0&1&0\\
0&0&0&1\\
0&0&0&0\\
0&0&0&0\end{array}\right)$ 
&$\left(\begin{array}{cccc}
0&0&0&0\\
0&0&0&0\\
0&0&0&0\\
0&0&0&0\end{array}\right)$\\\hline
&&&\\[-2.3ex]\cline{1-4}
&&&\\[-1ex]
\xymatrix{\textrm{\tiny{1}}\ar@(ul,ur)|{\bullet}&
\textrm{\tiny{2}}}
&\xymatrix{\textrm{\tiny{1}}\ar@(ur,ul)|{\bullet}&
\textrm{\tiny{2}}}
&\xymatrix{\textrm{\tiny{1}}&
\textrm{\tiny{2}}\ar@(ul,ur)|{\bullet}}
&\xymatrix{\textrm{\tiny{1}}&
\textrm{\tiny{2}}\ar@(ur,ul)|{\bullet}}
\\\cline{1-4}&&&\\[-1ex]
$\left(\begin{array}{cccc}
0&0&0&0\\
0&0&0&0\\
0&0&0&0\\
1&0&0&0\end{array}\right)$
&$\left(\begin{array}{cccc}
0&0&0&1\\
0&0&0&0\\
0&0&0&0\\
0&0&0&0\end{array}\right)$
&$\left(\begin{array}{cccc}
0&0&0&0\\
0&0&0&0\\
0&1&0&0\\
0&0&0&0\end{array}\right)$
&$\left(\begin{array}{cccc}
0&0&0&0\\
0&0&1&0\\
0&0&0&0\\
0&0&0&0\end{array}\right)$\\\cline{1-4}
&&&\\[-2.3ex]\cline{1-4}
&&&\\[-1ex]
\xymatrix{\textrm{\tiny{1}}\ar@(ul,ur)|{\bullet}&
\textrm{\tiny{2}}\ar@(ul,ur)|{\bullet}}
&\xymatrix{\textrm{\tiny{1}}\ar@(ul,ur)|{\bullet}&
\textrm{\tiny{2}}\ar@(ur,ul)|{\bullet}}
&\xymatrix{\textrm{\tiny{1}}\ar@(ur,ul)|{\bullet}&
\textrm{\tiny{2}}\ar@(ul,ur)|{\bullet}}
&\xymatrix{\textrm{\tiny{1}}\ar@(ur,ul)|{\bullet}&
\textrm{\tiny{2}}\ar@(ur,ul)|{\bullet}}\\\cline{1-4}
&&&\\[-1ex]
$\left(\begin{array}{cccc}
0&0&0&0\\
0&0&0&0\\
0&1&0&0\\
1&0&0&0\end{array}\right)$
&$\left(\begin{array}{cccc}
0&0&0&0\\
0&0&1&0\\
0&0&0&0\\
1&0&0&0\end{array}\right)$
&$\left(\begin{array}{cccc}
0&0&0&1\\
0&0&0&0\\
0&1&0&0\\
0&0&0&0\end{array}\right)$
&$\left(\begin{array}{cccc}
0&0&0&1\\
0&0&1&0\\
0&0&0&0\\
0&0&0&0\end{array}\right)$\\\cline{1-4}
\end{tabular}
$$

\end{example}

\subsection{Orbits in types B and D}\label{ssect:orth_B}

Let $G=\Or_n$, where $n\in\{2l,2l+1\}$. We denote by $B$ the standard Borel subgroup of $G$ and  consider  the algebra $\A(l)$ as discussed in Section \ref{sect:repAl}. We are interested in orthogonal representations of dimension vector $\mathbf{d}_\bullet=(1,2,...,l,n,l,...,2,1)$ by Lemma \ref{lem:Bijection}. 

As before, we begin with an example.

\begin{example}
Figure \ref{fig:ortho_exampleB4} shows the complete list of isomorphism classes of orthogonal representations 
in $\SRep(\A(2), \mathbf{d}_\bullet)^0$.

The first table shows the isomorphism classes of orthogonal representations where $\mathbf{d}_\bullet=(1,2,4,2,1)$, i.e. it corresponds to $\Or_4$.

The second table shows the isomorphism classes of orthogonal representations where $\mathbf{d}_\bullet=(1,2,5,2,1)$, i.e. it corresponds to $\Or_5$.

\begin{figure}[ht]
\begin{center}
\includegraphics[trim=130 445 70 145,clip,width=220pt]{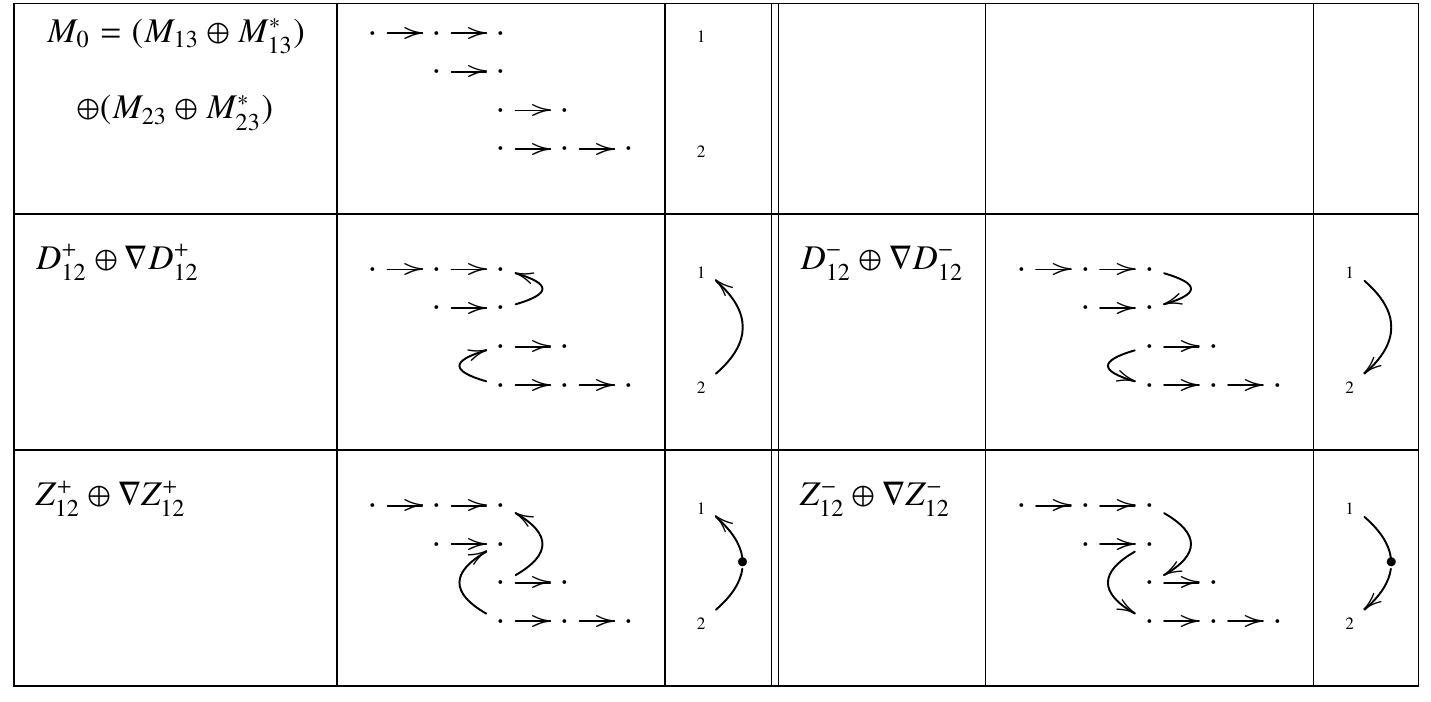}\end{center}\begin{center}
\includegraphics[trim=130 450 65 145,clip,width=220pt]{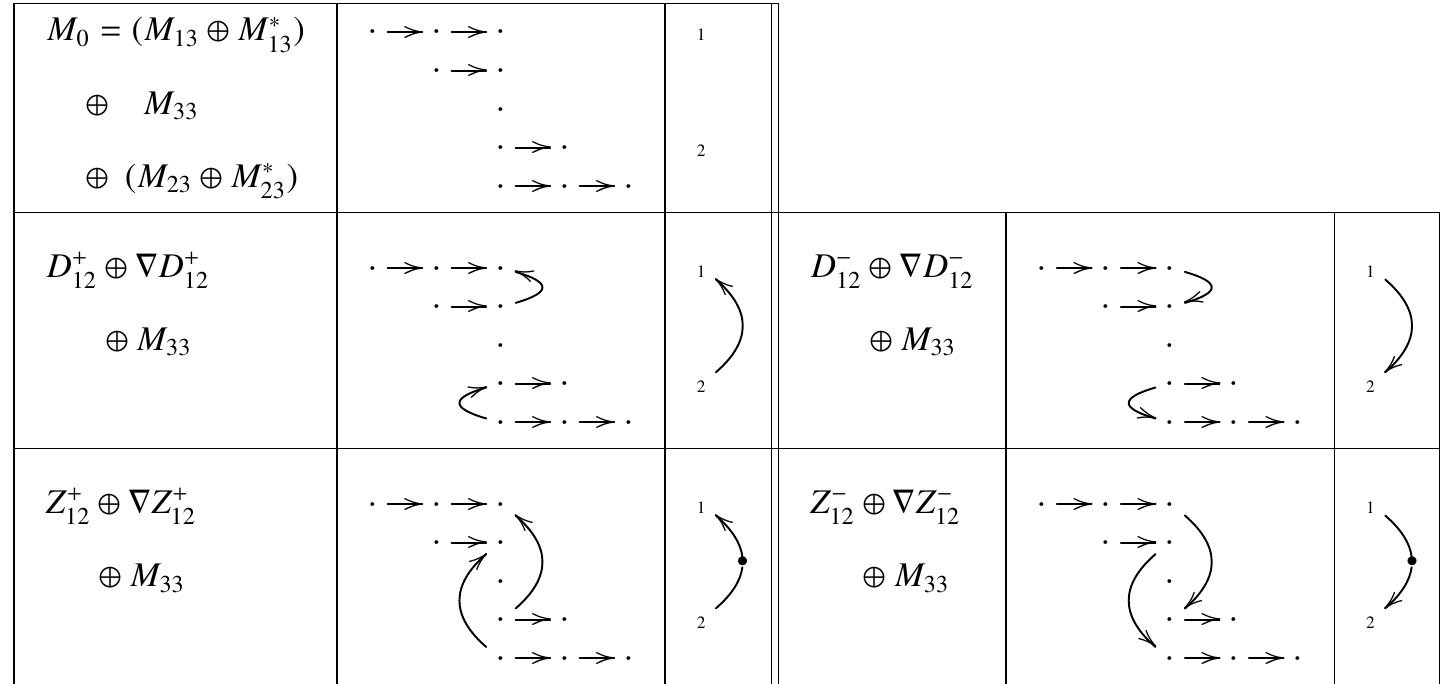}
\end{center}
\caption{Isomorphism classes of orthogonal representations 
in $\SRep(\A(2), \mathbf{d}_\bullet)^0$ and their combinatorial interpretation}\label{fig:ortho_exampleB4}
\end{figure}

\end{example}

This observation leads us to the following definition. 

\begin{definition}\label{Def:OOLP}
An \emph{orthogonal oriented link pattern} (oolp for short) of size $l$ consists of a  set of $l$ 
 vertices $1,2,\cdots l$ together with a collection of oriented arrows between these vertices and a collection of oriented arrows with dots between these vertices, such that
\begin{itemize}
\item[(B)] every vertex is touched by at most one arrow;
\item[(SpOr)] there are no loops of arrows without dots (that is, no arrows without dots from a vertex to itself).
\item[(Or)]\label{Or} there are no loops of arrows with dots (that is, no arrows with dots from a vertex to itself).
\end{itemize}
We denote by $\Oolp_{l}$ the set of oolps of size $l$.
\end{definition}
Thus, an oolp is a solp without loops.

As in the symplectic case, the parametrization of the Borel-orbits in $\N(2)$ follows straight away.

\begin{theorem}\label{Thm:ortho_paramB}
The $B$--orbits in the variety $\N(2)\subseteq \fo_n$, where $n\in\{2l,2l+1\}$, are in bijection with the set $\Oolp_l$ of oolps of size $l$. 
\end{theorem}
\begin{proof}

In a similar manner to Theorem \ref{Thm:symp_param_B}, there is a quite obvious bijection between the set $\Oolp_l$ of oolps  of size $l$ and the set of isoclasses of orthogonal representations of $\A(l)$ in $\SRep(\A(l),\mathbf{d}_\bullet)^0$ which maps an isomorphism class $[M]$ of an orthogonal representation to a particular subquiver of the coefficient quiver of $M$ induced by $M_\alpha$. 

In case $n$ is even, this is done analogously to the symplectic case. 

In case $n$ is odd, the middle vertex of this particular subquiver is always determined as a fixed point as visualized in Figure \ref{fig:ortho_exampleB4}. This is due to the fact that it corresponds to the direct summand $M_{\omega,\omega}$. The diagram representing $M_\alpha$ can thus be restricted to $2l$ vertices and we obtain the sought subquiver. This can be translated to an oolp as in the symplectic case again. In more detail, the translation of indecomposables to vertices / arrows in the oolp can be read off the following table - where also the translation to matrices in $\N(2)$ can be found.
\[\begin{tabular}{|c|c|c|}
\hline
(Part of) solp &  Indecomposable & (Part of) Matrix\\\hline
\xymatrix@C=4pt{
\textrm{\tiny{i}}} & $M_{i,l+1}\oplus M_{i,l+1}^\ast$ & $0$ \\\hline
&&\\[-1ex]
\xymatrix{\textrm{\tiny{i}}&
\textrm{\tiny{j}}\ar@/_1pc/[l]}&$D^+_{i,j}\oplus C^+_{i,j}$& $E_{i,j}-E_{j^\ast,i^\ast}$\\\hline
&&\\[-1ex]
\xymatrix{\textrm{\tiny{i}}\ar@/^1pc/[r]&
\textrm{\tiny{j}}}&$D^-_{i,j}\oplus C^-_{i,j}$& $E_{j,i}-E_{i^\ast,j^\ast}$\\\hline
&&\\[-2ex]
\xymatrix{\textrm{\tiny{i}}&
\textrm{\tiny{j}}\ar@/_1pc/[l]|{\bullet}}&$Z^+_{i,j}\oplus Z^+_{j,i}$& $E_{i,j^\ast}-E_{j,i^\ast}$\\\hline
&&\\[-2ex]
\xymatrix{\textrm{\tiny{i}}\ar@/^1pc/[r]|{\bullet}&
\textrm{\tiny{j}}}&$Z^-_{i,j}\oplus Z^-_{j,i}$&$E_{j^\ast,i}-E_{i^\ast,j}$ \\\hline
\end{tabular} \]
 As before, the claim follows from Lemma \ref{lem:Bijection}. 
\end{proof}

\begin{proposition}
Let $o_l$ be the cardinality of $\Oolp_l$. Then the sequence $\{o_l\}$ is determined by
\begin{itemize}
\item $o_0=1$,
\item $o_1=1$,
\item $o_l=o_{l-1}+4(l-1)o_{l-2}$.
\end{itemize} 
\end{proposition}
\begin{proof}
We divide the set $\Oolp_l$ into the subset of oolps where vertex $1$ is not touched by any arrow and its complement. 
\end{proof}
The sequence $1,1,5,13,73,281,1741,...$ which gives $\{o_l\}$ is classified  in OEIS as A115329 \cite{OEIS_olp}.

\begin{remark}
As before, we see that oolps are special oriented link patterns. As in the symplectic case, this fact also follows from \cite{MWZ}: if two orthogonal elements are conjugate under the Borel of $\GL_n$, then they are conjugate under the Borel of $\Or_n$, as well. 
\end{remark}
Again, we end the subsection by writing down the explicit classification of $B$-orbits in $\N(2)$ for the orthogonal group $\Or(4)$.
\begin{example}
Let $B\subset \Or(4)$, then the collection of oolps of size $2$ and the corresponding matrices in $\N(2)\subset \fo_4$ are given in the following table:
$$
\tiny
\begin{tabular}{|c|c|c|c|c|}
\hline&&&&\\[-1ex]
\xymatrix{\textrm{\tiny{1}}\ar@/^1pc/[r]&
\textrm{\tiny{2}}}
&\xymatrix{\textrm{\tiny{1}}&
\textrm{\tiny{2}}\ar@/_1pc/[l]}
&\xymatrix{\textrm{\tiny{1}}\ar@/^1pc/[r]|{\bullet}&
\textrm{\tiny{2}}}
&\xymatrix{\textrm{\tiny{1}}&
\textrm{\tiny{2}}\ar@/_1pc/[l]|{\bullet}}&\xymatrix{\textrm{\tiny{1}}&
\textrm{\tiny{2}}}
\\\hline
&&&&\\[-1ex]
$\left(\begin{array}{cccc}
0&0&0&0\\
1&0&0&0\\
0&0&0&0\\
0&0&-1&0\end{array}\right)$ 
& $\left(\begin{array}{cccc}
0&1&0&0\\
0&0&0&0\\
0&0&0&-1\\
0&0&0&0\end{array}\right)$ 
&$\left(\begin{array}{cccc}
0&0&0&0\\
0&0&0&0\\
1&0&0&0\\
0&-1&0&0\end{array}\right)$ 
&$\left(\begin{array}{cccc}
0&0&1&0\\
0&0&0&-1\\
0&0&0&0\\
0&0&0&0\end{array}\right)$ 
&$\left(\begin{array}{cccc}
0&0&0&0\\
0&0&0&0\\
0&0&0&0\\
0&0&0&0\end{array}\right)$\\\hline
\end{tabular}
$$

\end{example}
%
%

\section{Generalization to parabolic actions}\label{sect:ParGeneral}
Since there are only finitely many Borel orbits in $\N(2)$, we know that every parabolic $P$ acts finitely on the variety $\N(2)$, too. We aim to find explicit parametrizations for these parabolic orbits in case $P$ is a standard parabolic subgroup of $G$.

For the rest of the article, we consider $G$ to be either the symplectic group $\SP_{2l}$ in type C or the \textit{special} orthogonal group $\SO_{2l}$ (type D) or $\SO_{2l+1}$ (type B).

\subsection{Parabolic subgroups}\label{ssect:parSubgr}

We have seen before that the standard Borel subgroups  of $G$  equal the intersection of $G$ with the standard Borel subgroup of $\GL_n$. For standard parabolic subgroups, this is only true for types $B$ and $C$; in type $D$ there are standard parabolic subgroups $P\subset \SO_n$ which are not given as the intersection of a standard parabolic subgroup of $\GL_n$ with $\SO_n$.

Following Malle-Testermann in \cite[Chapter 12]{MaTe}, the parabolic subgroups of a classical group $G$ are in bijection to the so-called totally isotropic flags. These are flags
\[V_1\subset ...\subset V_k, \]
such that every two elements $v,w\in V_i$ vanish according to the bilinear form of $G$, namely $(v,w)=0$. Given such isotropic flag $F$, the parabolic $P$ equals the stabilizer of $F$ in $G$ which we denote by $\stab_G(F)$.

Let $F=V_1\subset ...\subset V_k$ be a totally isotropic flag, such that $P=\stab_G(F)$ includes the standard Borel $B$. Then we define $d_i:=\dim_K V_i$, the dimension vector $d_F=d_P=(d_1,...,d_k,n,d_k,...,d_1)$ and define the corresponding $\A(k)$-representation $M_F=M_P$ of dimension vector $d_F$ which represents the flag in a symmetric manner as usually. 

\begin{example}\label{Ex:parStandardGL}
Let $n=6$ and look at the totally isotropic flag

\[F_1:=(V_1=\langle e_1\rangle \subset V_2=\langle e_1,e_2,e_3\rangle) \]

Then $F_1$ corresponds to the $\A(2)$-representation $M_{F_1}$ of dimension vector $(1,3,6,3,1)$ with coefficient quiver
\begin{equation}
\xymatrix@R=3pt{
\bullet\ar^1[r]&\bullet\ar^1[r]&\bullet&&\\
&\bullet\ar^1[r]&\bullet&&\\
&\bullet\ar^1[r]&\bullet&&\\
&&\bullet\ar^{-1}[r]&\bullet&\\
&&\bullet\ar^{-1}[r]&\bullet&\\
&&\bullet\ar^{-1}[r]&\bullet\ar^{-1}[r]&\bullet\\
}
\end{equation}
Its stabilizer in $G$ is given as the upper-block standard parabolic subgroup of $\GL_6$ of block sizes $(1,2,2,1)$, intersected with $G$.

Let now $n=4$ and set
\[F_2:=(V_1=\langle e_1,e_3\rangle) \]

Then $F_2$ corresponds to the $\A(1)$-representation $M_{F_2}$ of dimension vector $(2,4,2)$ with coefficient quiver
\begin{equation}
\xymatrix@R=3pt{
\bullet\ar^1[r]&\bullet&\\
&\bullet\ar^{-1}[r]&\bullet\\
\bullet\ar^1[r]&\bullet&\\
&\bullet\ar^{-1}[r]&\bullet\\
}
\end{equation}
Its stabilizer in $\SO(4)$ is not given as the intersection of $\SO(4)$ with an upper-block standard parabolic subgroup of $\GL_4$.
 In fact,
\[\Lie \stab_{\SO(4)}(F_2)\cong \left\lbrace\left(\begin{array}{cccc}
a&c&d&0\\
0&b&0&-d\\
e&0&-b&-c\\
0&-e&0&-a
\end{array} \right)\mid a,b,c,d,e\right\rbrace\supset \mathfrak{b}.\]
\end{example}

Let us fix a standard parabolic $P$ now, together with the corresponding totally isotropic flag $F_P$ and the representation $M_P:= M_{F_P}$ of dimension vector $\mathbf{d}_P$. Our notions from Section \ref{sect:translation} generalize to this setup by looking at the variety $\SR_{\mathbf{d}_P}(\A(k))$ which is acted upon by the group 
\[\GL_{\sym}(P):=\GL(d_1)\times \GL(d_2)\times\cdots\times \GL(d_k)\times \Sym(n)\] where $\Sym(n)$ denotes either the symplectic or the orthogonal group on a vector space of dimension $n$, as before. Again, the open subset $\SR_{\mathbf{d}_P}(\A(k))^0$ corresponds to the full subcategory $\SRep(\A(k),\mathbf{d}_F)^0$ of $\SRep(\A(k),\mathbf{d}_P)$ of  maximal rank representations. 

We define the block vector $\mathbf{b}_P=(d_1,d_2-d_1,...,d_k-d_{k-1})$. In case of a standard parabolic which is induced by $\GL_n$, this vector gives the explicit structure of the parabolic as can be seen in Example \ref{Ex:parStandardGL}.

We obtain a translation as in Lemma \ref{lem:Bijection} which can be proven in the same way.

\begin{lemma}\label{lem:BijectionP}
There is a bijection between the isomorphism classes of symplectic/orthogonal $\A(k)$--representations in $\SRep(\A(k), \mathbf{d}_P)^0$ and symplectic/orthogonal $P$--orbits in $\N(2)$. 
\end{lemma}

\subsection{Orbits in type C}\label{ssect:symp_param_P}
We aim to classify the $P$-orbits in $\N(2)\subset \fsp_{n}$ by combinatorial objects and generalize the results of Subsection \ref{ssect:sympl_B} in a straight forward manner.

\begin{definition}
An \emph{enhanced symplectic oriented link pattern} (esolp. for short) of size $k$ of type $(b_1,...,b_k)$ consists of a  set of $k$ vertices $1,2,\cdots,k$ together with a collection of oriented arrows between different vertices, unoriented loops and oriented arrows with dots between arbitrary vertices. 

Denote by $x_i$  the number of sources of arrows, by $y_i$ the number of targets of arrows, by $z_i$ the number of unoriented loops and by $l_i$ the number of loops with dots at vertex $i$.
Then the following condition defines an esolp of such kind:
\begin{itemize}
\item[(Sp)]\label{Sp} $x_i+y_i+z_i-l_i\leq b_i$ holds for  each vertex $i$.
\end{itemize}
\end{definition}

\begin{theorem}\label{Thm:symp_param_P}
There is a natural bijection between the set of $P$-orbits in $\N(2)\subseteq \fsp_n$ and the set of esolps  of size $k$ and of type $(b_1,...,b_k)$.
\end{theorem}
\begin{proof}
In an analoguous manner to Theorem \ref{Thm:symp_param_B}, by Krull--Remak--Schmidt, there is a bijection between the set of esolps of size $k$ of type $(b_1,...,b_k)$ and the set of isomorphism classes of symplectic representations of $\A(k)$ in $\SRep(\A(k),\mathbf{d}_P)^0$ which maps an isomorphism class $[M]$ of a symplectic representation $M$ to the subquiver of the coefficient quiver of $M$ induced by $M_\alpha$ and then restricts the latter as follows: all vertices $i\in \{1,2,...,l\}$ which correspond to the same step in the flag $F$, say to step $V_s\setminus V_{s-1}$ are glued together to vertex $s$ of the esolp. This way, we obtain $2k$ vertices. We are again able to shrink them to $k$ vertices as the table in Figure \ref{fig:correspondenceParsympl} shows - via this correspondence, we obtain a bijection between symplectic direct sum decompositions and esolps.
\begin{figure}\label{fig:correspondenceParsympl}
\begin{center}
\begin{tabular}{|l|lc|}
\hline
Indecomposable & Multiplicity &\\\hline
$M_{s,k+1}\oplus M_{s,k+1}^\ast$ &$\begin{array}{l}
b_s-x_s-y_s+l_s \\
b_s-x_s-y_s \end{array}$ &$\begin{array}{l}
 \mathrm{(symplectic)}\\
\mathrm{(orthogonal)}
\end{array}$\\\hline
&&\\[1ex]
$D_{i,j}^-\oplus C_{i,j}^-$& number of arrows & \xymatrix{\textrm{\tiny{i}}\ar@/^1pc/[r]&
\textrm{\tiny{j}}} \\\hline
&&\\[0.5ex]
$D_{s,s}^+\oplus C_{s,s}^+ \cong D_{s,s}^-\oplus C_{s,s}^-$ & number of loops  &\xymatrix{\textrm{\tiny{s}}\ar@(ur,ul)@{-}[]}\\\hline
&&\\[-1ex]
$D_{i,j}^+\oplus C_{i,j}^+$& number of arrows  &\xymatrix{\textrm{\tiny{i}}&
\textrm{\tiny{j}}\ar@/_1pc/[l]}\\\hline
&&\\[-1ex]
 $Z_{i,j}^-\oplus Z_{j,i}^-$ & number of arrows  &\xymatrix{\textrm{\tiny{i}}\ar@/^1pc/[r]|{\bullet}&
\textrm{\tiny{j}}}\\\hline
&&\\[-1ex]
  $Z_{ij}^+\oplus Z_{j,i}^+$& number of arrows &\xymatrix{\textrm{\tiny{i}}&
\textrm{\tiny{j}}\ar@/_1pc/[l]|{\bullet}}\\\hline
&&\\[-1ex]
 $\begin{array}{ll}
Z_{s,s}^- & \mathrm{(symplectic)}\\
Z_{s,s}^-\oplus Z_{s,s}^-& \mathrm{(orthogonal)}
\end{array}$ & number of loops  &\xymatrix{\textrm{\tiny{s}}\ar@(ul,ur)|{\bullet}}\\\hline
&&\\[-1ex]
$\begin{array}{ll}
Z_{s,s}^+ & \mathrm{(symplectic)}\\
Z_{s,s}^+\oplus Z_{s,s}^+& \mathrm{(orthogonal)}
\end{array}$ & number of loops  &\xymatrix{\textrm{\tiny{s}}\ar@(ur,ul)|{\bullet}}\\\hline
\end{tabular}
\end{center}\caption{Correspondence of indecomposables and esolps for $i<j<k+1$ and $s<k+1$}
\end{figure}

By Lemma \ref{lem:Bijection}, the claim follows.
\end{proof}

Clearly, solps are special esolps: they are of size $l$ and of type $(1,...,1)$, such that we obtain the classification of Borel-orbits.

\begin{example}
Let $P$ be the symplectic parabolic subgroup of block sizes $(4,2)$, thus, $b_1=4$ and $b_2=2$. Then a symplectic representation of dimension vector $(4,6,12,6,4)$ is represented by a pattern of $12$ coloured vertices which represents the map $M_{\alpha}$ of the representation, for example by
\vspace{0.5cm}
$$\xymatrix@C=4pt{
\fcolorbox{black}{green!25}{$\textrm{\tiny{1}}^{\color{green!25}\ast}$}\ar@/^1.5pc/[rr]
&\fcolorbox{black}{green!25}{$\textrm{\tiny{2}}^{\color{green!25}\ast}$}\ar@/^2.5pc/[rrrrrrrrr]
&\fcolorbox{black}{green!25}{$\textrm{\tiny{3}}^{\color{green!25}\ast}$}
&\fcolorbox{black}{green!25}{$\textrm{\tiny{4}}^{\color{green!25}\ast}$}\ar@/^1.5pc/[rrr]
&\fcolorbox{black}{blue!25}{$\textrm{\tiny{5}}^{\color{blue!25}\ast}$}
&\fcolorbox{black}{blue!25}{$\textrm{\tiny{6}}^{\color{blue!25}\ast}$}
&\fcolorbox{black}{blue!25}{$\textrm{\tiny{6}}^\ast$}
&\fcolorbox{black}{blue!25}{$\textrm{\tiny{5}}^\ast$}
&\fcolorbox{black}{green!25}{$\textrm{\tiny{4}}^\ast$}\ar@/^1.5pc/[lll]
&\fcolorbox{black}{green!25}{$\textrm{\tiny{3}}^\ast$}\ar@/_1.5pc/[rr] 
&\fcolorbox{black}{green!25}{$\textrm{\tiny{2}}^\ast$}
&\fcolorbox{black}{green!25}{$\textrm{\tiny{1}}^\ast$}}$$

\vspace{0.5cm}
This pattern corresponds to the indecomposable direct sum decomposition 
\[(D_{1,1}^+\oplus C_{1,1}^+)\oplus Z_{1,1}^+ \oplus (Z_{1,2}^-\oplus Z_{1,2}^+)\oplus (M_{2,3}\oplus M_{2,3}^{\ast}) \]
of a representation of the quiver $\Q_2$. The corresponding esolp is given by

\[ \xymatrix{\fcolorbox{black}{green!25}{\textrm{1}}\ar@(l,ul)|{\bullet}\ar@(dl,dr)@{-}[]\ar@/^1pc/[r]|{\bullet}&
\fcolorbox{black}{blue!25}{\textrm{2}}}\]

\vspace{0.5cm}
We see that $b_1-x_1-y_1+l_1=4-4-1+1=0$ and $b_2-x_2-y_2+l_2=2-0-1+0=1$; these equations show that the multiplicity of the indecomposable  $M_{1,3}\oplus M_{1,3}^{\ast}$ is zero and the multiplicity of the indecomposable $M_{2,3}\oplus M_{2,3}^{\ast}$ is one.
\end{example}

\subsection{Orbits in type B and D}\label{ssect:parparameven}

Let us consider a standard parabolic subgroup of $\SO_{2l+1}$ or $\SO_{2l}$ together with its totally isotropic flag $F$ now. Then the classification of orbits is done analogously to the symplectic case, but in type D we have to be careful, since not every parabolic is obtained from a standard flag.

\begin{definition}
An \emph{enhanced orthogonal oriented link pattern} (esolp. for short) of size $k$ of type $(b_1,...,b_k)$ consists of a  set of $k$ vertices $1,2,\cdots,k$ together with a collection of oriented arrows and oriented arrows with dots between these vertices. 

Denote by $x_i$  the number of sources of arrows, by $y_i$ the number of targets of arrows and by $z_i$ the number of unoriented loops at vertex $i$.
Then the following condition defines an eoolp of such kind:
\begin{itemize}
\item[(SO)]\label{SO} $x_i+y_i+2z_i\leq b_i$ holds for  each vertex $i$.
\end{itemize}
\end{definition}

An esolp is an eoolp, if $b_i=1$ implies that there are no loops at vertex $i$.

The classification of parabolic orbits follows similarly to the considerations in \ref{Thm:symp_param_P}.
\begin{theorem}\label{Thm:ortho_paramP}
There is a natural bijection between the set of $P$-orbits in $\N(2)\subseteq \fg$ and the set of eoolps of size $k$ and of type $(b_1,...,b_k)$.
\end{theorem}

\begin{proof}
In type B, every standard parabolic subgroup is given as the stabilizer of a totally isotropic \textbf{standard} flag.  Given a $P$-orbit $P.x$ in $\N(2)$ we look at the representation $M_x$ via Lemma \ref{lem:Bijection}. Consider the coefficient quiver of $M_x$ and restrict it to those $n$ vertices which represent the loop $M_{\alpha}$. Then the eoolp is obtained by identifying vertices according to the flag and then restricting to the vertices $1,...,k$ via the correspondence depicted in Figure \ref{fig:correspondenceParsympl} - similarly as in the symplectic case. Note that the vertex $k+1$ is  redundant, since it always represents the indecomposable $M_{\omega,\omega}$ (as in the proof of Theorem \ref{Thm:ortho_paramB}). 

In type D, there are two cases to consider.

\begin{enumerate}
\item The flag $F$ is a standard flag, such that $P$ is the intersection of a standard parabolic of $\GL_n$ with $\SO_n$.

In this case, the classification of the orbits is done similarly as in Theorem \ref{Thm:symp_param_P} with the correspondence table given in Figure \ref{fig:correspondenceParsympl}.

\item For every missing standard parabolic, the flag $F$ is \textbf{not} a standard flag and $P$ does not equal the intersection of a standard parabolic of $\GL_n$ with $\SO_n$.

In this case, the classification is still done in the same mannner, since we deal with the Krull-Remak-Schmidt decomposition into unique indecomposables.  Thus, Figure \ref{fig:correspondenceParsympl} is still a valid correspondence for the bijection between orbits and eoolps and the combinatorial description via eoolps is just only dependent on the block vector of the flag $(b_1,...,b_k)$.\qedhere
\end{enumerate}
\end{proof}
Clearly, oolps are special eoolps, they are of size $l$ and of type $(1,...,1)$, such that we obtain the classification of Borel-orbits.

Even though the classification of orbits for parabolics of non-standard totally isotropic flags in type D is done in a similar manner as in the other cases, these parabolics are more difficult to handle - especially if one is interested in partiular representative matrices of the $P$-orbits in $\N(2)$ - and we give them some more attention by calculating an example.

\begin{example}
Let us consider $n=6$ and the totally isotropic flag

\[F=(V_1=\langle e_1\rangle \subset V_2=\langle e_1,e_2,e_4\rangle). \]

Then the coefficient quiver of $M_F$ is
\begin{equation}
\xymatrix@R=3pt{
\bullet\ar^1[r]&\bullet\ar^1[r]&
\bullet&&\\
&\bullet\ar^1[r]&\bullet&&\\
M_F=&&\bullet\ar^{-1}[r]&\bullet&\\
&\bullet\ar^1[r]&\bullet&&\\
&&\bullet\ar^{-1}[r]&\bullet&\\
&&\bullet\ar^{-1}[r]&\bullet\ar^{-1}[r]&\bullet\\}
\end{equation}
and the isomorphism classes of indecomposables are listed in the table in Figure \ref{fig:exampleP}.

\begin{figure}[ht]
\begin{center}\includegraphics[trim=130 105 25 145,clip,width=270pt]{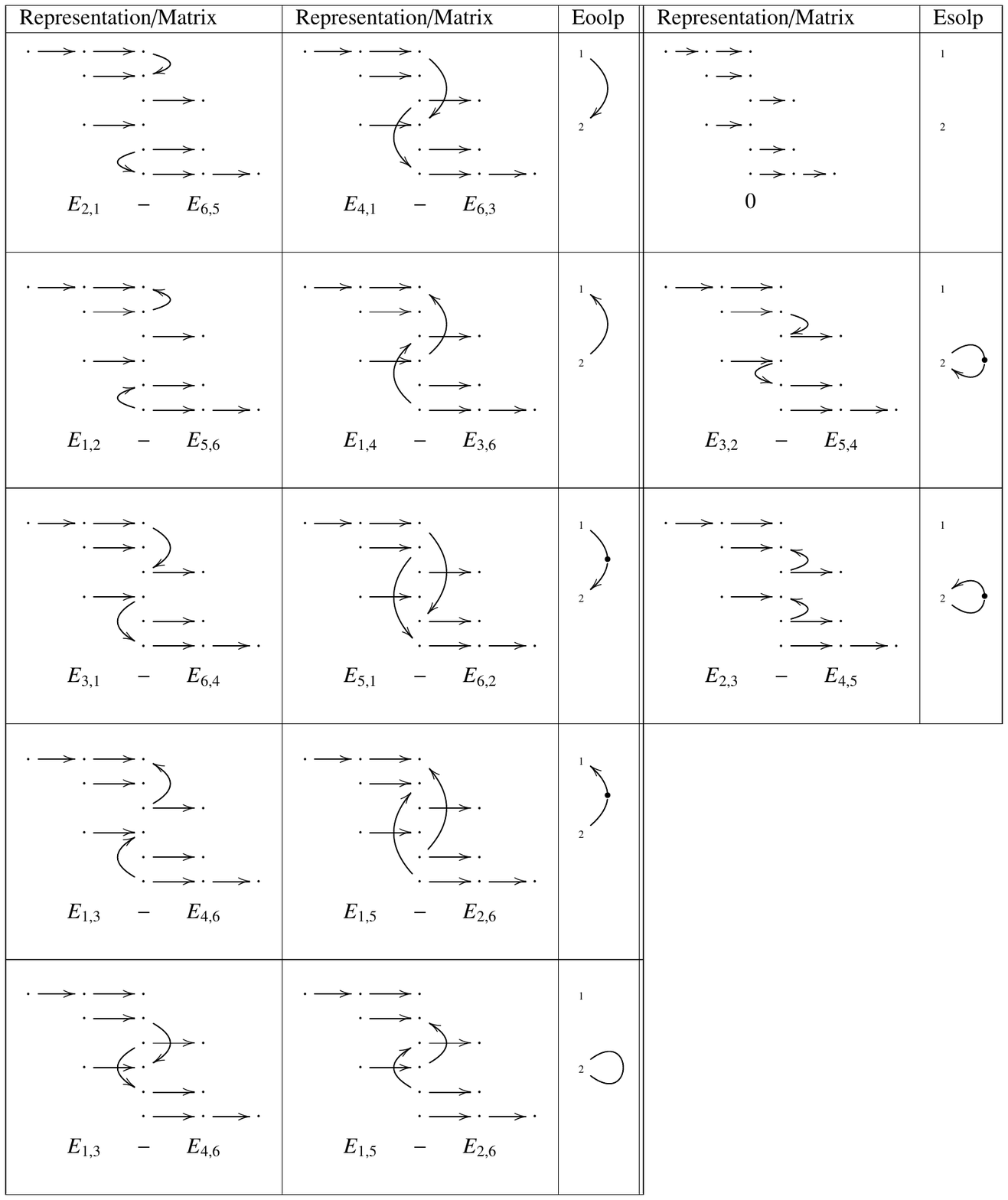}\end{center}\caption{Isomorphism classes of orthogonal representations 
in $\SRep(\A(2), \mathbf{d}_P)^0$ and their combinatorial interpretation}\label{fig:exampleP}
\end{figure}

 \end{example}

\section{Restriction to the nilradical}
If we restrict a parabolic action on $\N(2)$ to the nilradical $\mathfrak{n}(2)$ of $2$-nilpotent upper-triangular matrices in the given Lie algebra, then we still have a parabolic action. The parametrization of the orbits can be obtained from our parametrizations of Sections \ref{sect:paramOrbits} and \ref{sect:paramOrbits} straight away. The action of the Borel subgroup in the symplectic case is parametrized in detail in \cite{BaMe}, where Barnea and Melnikov also derive a description of the orbit closures and look at applications to orbital varieties in detail. In \cite{GMMP}, Gandini, Maffei, M\"oseneder Frajria and Papi consider the more general approach of $B$-stable abelian subalgebras of the
nilradical of $\mathfrak{b}$ in which they parametrize the $B$-orbits and describe their closure
relations.

\begin{definition}
A \emph{symplectic link pattern} (slp for short) of size $k$ is a symplectic oriented link pattern, such that every arrow is oriented \textit{from right to left} and every loop is oriented counterclockwise. 
In the same way, the natural notion of \emph{orthogonal link pattern} (orlp for short)
 \emph{enhanced symplectic} and \emph{enhanced orthogonal link pattern} is obtained.
\end{definition}
The sets of (enhanced) symplectic and (enhanced) orthogonal link patterns are obtained by taking all such oriented patterns and deleting the orientation. 
\begin{corollary}
There is a bijection between the parabolic orbits in the  symplectic nilradical $\mathfrak{n}(2)$ (or orthogonal $\mathfrak{n}(2)$, resp.) and the set of enhanced symplectic (or orthogonal, resp.) link patterns.
\end{corollary}
\begin{proof}
The fact that all arrows are oriented from right to left corresponds equivalently to the nilpotent map at the loop being upper-triangular. Thus, these are directly the nilpotent elements contained in the nilradical of the particular Lie algebra.
\end{proof}

We end the section with giving an example.
\begin{example}
For $l=2$, we consider the Borel-action. Figures \ref{fig:exampleB} and \ref{fig:ortho_exampleB4}  show the possible patterns. Taking away the orientation is equivalent to only considering upper-triangular matrices. Thus, the following table gives a complete list of representatives of orbits for type C. The patterns which are marked green give a complete list of orthogonal $B$-orbits.

$$
\begin{tabular}{|c|c|c|c|c|c|}
\hline\cellcolor{green!25}&\cellcolor{green!25}&\cellcolor{green!25}&&&\\[-1ex]
\cellcolor{green!25}\xymatrix@C=4pt{
\textrm{\tiny{1}}&\textrm{\tiny{2}}}&\cellcolor{green!25}\xymatrix@C=4pt{
\textrm{\tiny{1}}\ar@/^1.5pc/@{-}[r]&\textrm{\tiny{2}}
}&
\cellcolor{green!25}\xymatrix@C=4pt{
\textrm{\tiny{1}}\ar@/^1.5pc/@{-}[r]|{\bullet}&\textrm{\tiny{2}}}
&\xymatrix@C=4pt{
\textrm{\tiny{1}}\ar@(ul,ur)@{-}|{\bullet}&\textrm{\tiny{2}}}
&\xymatrix@C=4pt{
\textrm{\tiny{1}}&\textrm{\tiny{2}}\ar@(ul,ur)@{-}|{\bullet}}
&
\xymatrix@C=4pt{
\textrm{\tiny{1}}\ar@(ul,ur)@{-}|{\bullet}&\textrm{\tiny{2}}\ar@(ul,ur)@{-}|{\bullet}
}
\\\hline
\end{tabular}
$$
\end{example}

\end{document}